\documentclass[12pt]{article}
\usepackage[top=2cm, bottom=2cm, left=2cm, right=2cm] {geometry}
\usepackage[arrow, matrix]{xy}
\usepackage{amsmath}
\usepackage{dsfont}
\usepackage{cancel}
\usepackage{amssymb}
\usepackage{amsthm}
\usepackage[mathscr]{eucal}
\usepackage{graphicx}
\usepackage{tikz}

\usepackage{lineno}

\input{epsf}
\theoremstyle{definition}
\newtheorem{definition}{Definition}
\newtheorem{theorem}[definition]{Theorem}
\newtheorem{remark}[definition]{Remark}
\newtheorem{proposition}[definition]{Proposition}
\newtheorem{lemma}[definition]{Lemma}
\newtheorem{corollary}[definition]{Corollary}
\newtheorem{example}[definition]{Example}
\newcounter{enumctr}

\newcommand{\N}{\mathbb{N}}
\newcommand{\R}{\mathbb{R}}

\newcommand{\rT}{\mathrm{T}}
\newcommand{\Z}{\mathbb{Z}}

\newcommand{\eps}{\varepsilon}

\renewcommand{\phi}{\varphi}

\DeclareMathOperator{\diag}{diag}

\newcommand{\cM}{\mathcal{M}}

\begin{document}
\title{\vspace*{-10mm}
A new concept of local metric entropy for\\ finite-time nonautonomous dynamical systems}
\author{Luu Hoang Duc\\
Institute of Mathematics, VAST, Vietnam
\& \\
Center for Dynamics,
Technische Universit\"at Dresden\\
{\tt lhduc@math.ac.vn, hoang\_duc.luu@tu-dresden.de}\\[3ex]
Stefan Siegmund \\
Institute for Analysis \& Center for Dynamics\\
Technische Universit\"at Dresden\\
{\tt stefan.siegmund@tu-dresden.de}\\[3ex]
}

\maketitle
\begin{abstract}
We introduce a new concept of finite-time entropy which is a local version of the classical concept of metric entropy. Based on that, a finite-time version of Pesin's entropy formula and also an explicit formula of finite-time entropy for $2$-D systems are derived. We also discuss about how to apply the finite-time entropy field to detect special dynamical behavior such as Lagrangian coherent structures.
\end{abstract}
\vspace{1ex} \noindent {\bf Key words and phrases:} finite-time metric entropy (FTME), finite-time Lyapunov exponents (FTLE), Pesin's formula, Lagrangian coherent structures (LCS)
\\[1ex]

%


\section{Introduction}

\emph{Metric} or \emph{measure-theoretic entropy} for a transformation was introduced by Kolmogorov and Sinai in 1959; the ideas go back to Shannon’s information theory (see Young \cite{young} for a well-written survey and the references therein). Concepts and results on metric entropy can be formulated in general settings \cite{young}. We recall some basic facts for the special situation of a map $T : M \rightarrow M$ on a compact set $M \subset \R^n$ which is measure-preserving w.r.t.\ the Lebesgue measure $\mu$, i.e.\ for every Borel set $A$, $T^{-1} A$ is also a Borel set, and $\mu(A) = \mu(T^{-1} A)$. The \emph{metric entropy} $h_\mu(T)$ of $T$ w.r.t.\ $\mu$ can be defined as the supremum over all entropies of finite partitions of $M$. However, we are particularly interested in a local characterization of metric entropy which goes back to Bowen \cite[Definition 6 \& Proposition 7]{bowen} (see also \cite{brinkatok83, young}) and was then generalized by Thieullen \cite{thieullen1, thieullen2}. According to Thieullen \cite{thieullen1}, for any $x \in M$, $n \in \N, \alpha \geq 0$ and $\eps >0$, define
\begin{equation}\label{autonomousstableset}
B^\alpha(x,n,\eps) := \{y \in X: \sup_{0\leq i\leq n} \| T^ix - T^i y \| e^{i \alpha}< \eps \}
\end{equation}
and the local quantities 
\begin{equation}\label{local.entropy}
\begin{array}{rcl}
  \overline{h}_\mu(\alpha,T,x) 
  &:=& 
  \lim_{\eps \to 0} \limsup_{n\to \infty} -\frac{1}{n} \log \mu\big(B^\alpha(x,n,\eps)\big) ,
\\[1ex]
  \underline{h}_\mu(\alpha,T,x) 
  &:=&
  \lim_{\eps \to 0} \liminf_{n\to \infty} -\frac{1}{n} \log \mu\big(B^\alpha(x,n,\eps)\big) .
\end{array}
\end{equation}


If $T$ is a $C^2$ diffeomorphism then $\overline{h}_\mu(\alpha,T,x) = \underline{h}_\mu(\alpha,T,x) =: h_\mu(\alpha,T,x)$ and we call $h_\mu$ the local $\alpha$-entropy. Moreover, if $\alpha=0$ and $T$ is ergodic then $h_\mu(\alpha,T,x)= h_\mu(T)$ for almost every $x\in M$ \cite{brinkatok83}. Local $\alpha$-entropy also has the interpretation of being the rate of loss of information on nearby orbits. Its relation to the Lyapunov exponents of $T$ are described by Pesin's formula \cite{pesin77,qian,thieullen1}. Let $\lambda_n(x) \leq \ldots \leq \lambda_1(x)$ denote the Lyapunov exponents of $T$ at the point $x$. Write $a^+ = \max(a,0)$.


\begin{theorem}[Pesin's formula for $\alpha$-entropy \cite{thieullen1}]\label{alphaentropy}
Let $T: M \to M$ be a $C^2$ diffeomorphism preserving the Lebesgue measure $\mu$. Then for almost every $x\in M$
\begin{equation*}
  h_\mu(\alpha,T,x) 
  =  
  \left\{ 
    \begin{array}{ll} 
      \sum_{i=1}^n (\lambda_i(x)+\alpha)^+ 
      &\text{if}\  0 \leq \alpha \leq - \lambda_n(x) \nonumber
      \,,
  \\[1ex]
    n \alpha
    &\text{if}\  \alpha \geq -\lambda_n(x) 
      \,.
    \end{array}
  \right.
\end{equation*}
\end{theorem}
In Section 2 we introduce a new concept of \emph{finite-time metric entropy (FTME)} which is motivated by the local quantities \eqref{local.entropy}. The notion of FTME is defined in Definition \ref{defFTME} w.r.t.\ the Lebesgue measure in the setting of \emph{nonautonomous dynamical systems (NDS)} on fibre bundles $X \subseteq \R^n \times J$ over compact subsets $J$ of $\R$ and can be easily adapted to NDS on Riemannian manifolds with measures which are equivalent to the Riemannian measure. Our concept of FTME is related to, but formally different from the probabilistic concept of finite-time entropy (FTE) introduced by Froyland and Padberg-Gehle \cite{froyland} which is based on the concept of differential entropy for a smoothed transfer operator (see Remark \ref{remFTME}(e) for a comparison). FTME of a nonlinear NDS can be expressed by FTME of its linearization (Theorem \ref{FTEcompute}) and is proportional to the measure of the intersection of ellipsoids which are the preimages of balls under the linearized NDS (Corollary \ref{thm-ellipsoid}).

In Section 3 we prove a finite-time version of Pesin's formula from Theorem \ref{alphaentropy} which relates the FTME to the sum of finite-time Lyapunov exponents which are not less than the weight factor $\alpha$. For one and two-dimensional NDS an exact formula is given in \eqref{entropy1d} and Proposition \ref{Pesin2D}. The main approximation result which holds in arbitrary dimensions is contained in Theorem \ref{thm2}.

In Section 4 we introduce Lagrangian coherent structures (LCS) based on the new notion of FTME. For a discussion of LCS based on finite-time Lyapunov exponents see e.g.\ \cite{haller1, haller1b} and the references therein. In order to formulate Theorem 2 in \cite{haller1b} for two-dimensional differential equations (see also \cite{haller1, haller1b} for arbitrary dimensions), consider a planar differential equation $\dot x = f(t,x)$, $t \in [t_0,t_0+T]$, $x \in \R^2$, for some $T > 0$ with solution $\phi(t,s)x_0$ which takes the initial value $x_0$ at $t = s$. Let $\Lambda_1(t_0,x_0,T) \geq \Lambda_2(t_0,x_0,T)$ and $\xi_1(t_0,x_0,T), \xi_2(t_0,x_0,T)$ denote the singular values and singular vectors of $\Phi_{x_0}(t_0+T,t_0) := D\phi(t_0+T,t_0)x_0$, respectively, i.e.\ $\Phi_{x_0}(t_0+T,t_0)^\top \Phi_{x_0}(t_0+T,t_0)  \xi_i(t_0,x_0,T) = \Lambda_i(t_0,x_0,T) \xi_i(t_0,x_0,T)$. The \emph{finite-time Lyapunov exponents (FTLE)} are defined by $\lambda_i(t_0,x_0,T) := \frac{1}{T} \log \Lambda_i(t_0,x_0,T)$. Note that $\lambda_1(t_0,x_0,T) \geq  \lambda_2(t_0,x_0,T)$ (in contrast to the reversed order in \cite{haller1b}).
Consider a smooth compact curve $\cM(t) \subset \R^2$ at time $t_0$ which is mapped by the solution map into a time-evolving curve $\cM(t) = \phi(t,t_0) \cM(t_0)$. For each $x_0 \in \cM(t_0)$ denote the tangent space of $\cM(t_0)$ at $x_0$ by $T_{x_0} \cM(t_0)$.

\begin{theorem}[LCS and Weak LCS in Two Dimensions {\cite[Thm.\ 2]{haller1b}}]\label{thmFTLE}\hfill\\
(i) $\cM(t)$ is a repelling weak LCS (WLCS) over $[t_0,t_0+T]$ if and only for all $x_0 \in \cM(t_0)$:
\vspace*{-1ex}
\begin{itemize}
  \item[1.] $\Lambda_2(t_0,x_0,T) \neq \Lambda_1(t_0,x_0,T) > 1$
  \item[2.] $\xi_1(t_0,x_0,T) \perp T_{x_0} \cM(t_0)$
  \item[3.] $\langle \nabla \Lambda_1(t_0,x_0,T), \xi_1(t_0,x_0,T) \rangle = 0$
\end{itemize}

(ii) $\cM(t)$ is a repelling LCS over $[t_0,t_0+T]$ if and only if:
\vspace*{-1ex}
\begin{itemize}
  \item[1.] $\cM(t)$ is a repelling WLCS over $[t_0,t_0+T]$
  \item[2.] $\langle \xi_1(t_0,x_0,T), \nabla^2 \Lambda_1(t_0,x_0,T) \xi_1(t_0,x_0,T) \rangle < 0$
\end{itemize}
\end{theorem}
In constrast to emphasizing the normal direction of $\cM$ in condition 2 of Theorem \ref{thmFTLE}(i), we introduce a stretching rate along the direction of the vector field in Section 4 and use this as a (local in time and space) weight factor for normalizing the exponential growth rates. This weight factor leads to a loss of frame-independence (cp.\ Remark \ref{rem:LCS}), but is chosen adequately so that we can show in Theorem \ref{autonomentropy} and explicitely for a family of nonlinear autonomous equations in Example \ref{exlinear} and even for linear systems in Example \ref{parabola}, that the ridge and trough-like structures of this weighted FTME field are able to recover stable and unstable manifolds. See also \cite{DoanEtAl2011, DucSiegmund2011} for alternative approaches to finite-time spectrum and hyperbolicity.


\section{Finite-time entropy}

Let $J \subseteq \R$ and $\big(X(t)\big)_{t \in J} \subseteq \R^n$ be a family of subsets of $\R^n$ indexed by $J$. Then $X := \{(t,x) \in J \times \R^n : x \in X(t)\}$ is a (trivial) fibre bundle over the base space $J$.

A continuous map $\phi : J \times X \rightarrow X$ is called
a \emph{nonautonomous dynamical system (NDS) on $X$ over $J$}, if for $t,u,s \in J$ and $x \in X(s)$ the properties $\phi(s,s,x) = (s,x)$ and $\phi(t,u,\phi(u,s,x)) = (t,\phi(t,s,x))$ hold. For ease of notation we identify $\phi$ with the two-parameter family of maps $\phi(t,s) = \phi(t,s,\cdot) : X(s) \rightarrow X(t) \subseteq \R^n$, $t,s \in J$, and the defining properties read as
\[
  \phi(s,s) x = x
  \quad \text{and} \quad
  \phi(t,u) \circ \phi(u,s) x = \phi(t,s) x
  .
\]
Obviously $\phi(t,s)^{-1} = \phi(s,t)$. If $J$ is compact, then $\phi$ is called \emph{finite-time nonautonomous dynamical system (FTNDS)}. If for all $t,s\in J$ the maps $\phi(t,s): X(t)\to X(s)$ are $C^k$ and all derivatives depend continuously on $t,s\in J$, we say that $\phi$ is $C^k$. We write $|J| := \max{J} - \min{J}$.

Note that the term \emph{nonautonomous dynamical system (NDS)} is sometimes used in slightly different contexts (see e.g.\ \cite{BergerSiegmund2003} and the references therein), either refering to a cocycle (with time $t$ measuring the time which elapsed since the starting time) or a process (with time $t$ measuring absolute time).

\begin{example}\label{example-NDS}
(a) A homeomorphism $T : M \rightarrow M$ on $M \subseteq \R^n$ generates an NDS $\phi(t,s)x := T^{t-s} x$ on $\Z \times M$ over $\Z$.

(b) Let $D \subseteq \R \times \R^n$ be open and $f \in C^k(D,\R^n)$ for some $k \in \N$. For $(t_0,x_0) \in D$ let $\phi(\cdot, t_0,x_0)$ denote the solution of the initial value problem
\[
  \dot x = f(t,x)
  , \quad
  x(t_0)=x_0
  .
\]
If for an arbitrary $J \subseteq \R$ and a family $\big(X(t)\big)_{t \in J}$ of subsets of $\R^n$ each map $\phi(t,s,\cdot) : X(s) \rightarrow X(t)$, for $t, s \in J$, is well-defined, then $\phi$ is an NDS on $X$ over $J$ and is $C^k$.

(c) In the setting of (b), let $\Phi_{x_0}(t,s)$ denote the solution matrix of the linearization $\dot v = D_x f(t,\phi(t,s,x_0)) v$ which satisfies $\Phi_{x_0}(s,s) = I_{n \times n}$ for $(s,x_0) \in X$, $t \in J$. Then $D \phi(t,s) x_0 = \Phi_{x_0}(t,s)$ and $\Phi_{x_0}$ is a linear NDS on $X$ over $J$.
\end{example}

Let $\|\cdot\|$ denote the Euclidean norm on $\R^n$. For a finite-time NDS $\phi$ on $X$ over a compact $J$ we want to measure the distance of orbits ${\cal O}(t_0,x) := \big\{\big(t,\phi(t,t_0)x\big) : t \in J \big\} \subset X$ to other orbits ${\cal O}(t_0,y)$ and thereto introduce a parametrized family $d^\alpha : X \times_{J} X \rightarrow J \times \R_0^+$, $(t_0,x,y) \mapsto (t_0,d_{t_0}^\alpha(x,y))$, of fibre metrics on the fibre product $X \times_{J} X := \big( (X(t) \times X(t)\big)_{t \in J}$, by defining for $\alpha \in \R$ the \emph{weighted orbit metric}
\begin{equation}\label{weightedmetric}
  d_{t_0}^\alpha(x,y)
  :=
  \sup_{t\in J} \|\varphi(t,t_0)y - \varphi(t,t_0)x\| e^{-\alpha(t-t_0)}
  .
\end{equation}
The dependency of $d_{t_0}^\alpha = d_{t_0}^{\phi,J,\alpha}$ on $\phi$ and $J$ is sometimes denoted in the superscript.
Using the fact that $(t,x) \in {\cal O}(t_0,x_0) \Leftrightarrow x = \phi(t,t_0)x_0$, it is easy to see that
\[
  d_{t_0}^\alpha(x_0,y_0) \leq \eps
  \quad \Leftrightarrow \quad
  \forall (t,x) \in {\cal O}(t_0,x_0), (t,y) \in {\cal O}(t_0,y_0) : \|x - y\| \leq \eps e^{-\alpha t_0} \cdot e^{\alpha t}
  .
\]
The balls w.r.t.\ the orbit metric are denoted by
\[
  B_{t_0}^{\alpha}(x_0,\eps)
  :=
  \{x \in X(t_0) : d_{t_0}^{\alpha}(x,x_0) \leq \eps\}
  \qquad \text{for } (t_0,x_0) \in X, \eps \geq 0
  .
\]
Since $d_{t_0}^{\alpha}(x,x_0) \leq \|x - x_0\|$, obviously $B_{t_0}^{\alpha}(x_0,\eps) \subseteq B(x_0,\eps) := \{x \in \R^n : \|x - x_0\| \leq \eps\}$.

\begin{proposition}[Properties of orbit metric]\label{thm-orbit-metric}
Let $\phi$ be an FTNDS on $X$ over $J$.

(i) \emph{Invariance:}  For $(t_0,x_0) \in X, \eps \geq 0$ and arbitrary $t \in J$
\begin{equation}\label{property01}
  \phi(t,t_0) B_{t_0}^{\alpha}(x_0,\eps)
  =
  B_{t}^{\alpha}\big(\varphi(t,t_0)x_0,\eps e^{\alpha(t-t_0)}\big).
\end{equation}

(ii) \emph{Monotonicity:} For $(t_0,x_0) \in X$, $0 \leq \eps_1 \leq \eps_2$ and $t_0 \in J_2 \subseteq J_1 \subseteq J$
\begin{equation}\label{property02}
  B_{t_0}^{J_1,\alpha}(x_0,\eps_1)
  \subseteq
  B_{t_0}^{J_2,\alpha}(x_0,\eps_2)
  .
\end{equation}
\end{proposition}

\begin{proof}
(i) We rewrite $B_{t_0}^{\alpha}(x_0,\eps)$ in the following form
\begin{eqnarray}\label{stableset}
  B_{t_0}^{\alpha}(x_0,\eps)
  &=&
  \big\{x\in X(t_0) : \|\phi(t,t_0)x - \phi(t,t_0)x_0\| \leq \eps e^{\alpha(t-t_0)} \;\forall t\in J\big\}
\nonumber\\
  &=&
  \big\{x\in X(t_0) : \phi(t,t_0)x \in B\big(\phi(t,t_0)x_0, \eps e^{\alpha(t-t_0)}\big) \;\forall t\in J\big\}
\nonumber\\
  &=&
  \big\{x\in X(t_0) : x \in \phi(t,t_0)^{-1} B\big(\phi(t,t_0)x_0, \eps e^{\alpha(t-t_0)}\big) \;\forall t\in J\big\}
\nonumber\\
  &=&
  \bigcap_{t \in J} \phi(t,t_0)^{-1} B\big(\phi(t,t_0)x_0, \eps e^{\alpha(t-t_0)}\big)
  .
\end{eqnarray}
To derive \eqref{property01}, we observe that for $t \in J$
\begin{eqnarray*}
  \phi(t,t_0) B_{t_0}^{\alpha}(x_0,\eps)
  &=&
  \phi(t,t_0)  \bigcap_{s \in J} \phi(s,t_0)^{-1} B\big(\phi(s,t_0)x_0, \eps e^{\alpha(s-t_0)}\big)
\\
  &=&
  \bigcap_{s \in J} \phi(t,t_0) \phi(s,t_0)^{-1} B\big(\phi(s,t_0)x_0, \eps e^{\alpha(s-t_0)}\big)
\\
  &=&
  \bigcap_{s \in J} \phi(s,t)^{-1} B\big(\phi(s,t) \phi(t,t_0)x_0, \eps e^{\alpha(t-t_0)} e^{\alpha(s-t)}\big)
\\
  &=&
  B_{t}^{\alpha}(\phi(t,t_0)x_0, \eps e^{\alpha(t-t_0)})
  .
\end{eqnarray*}

(ii) If $\eps_1 \leq \eps_2$ and $J_2 \subseteq J_1$ then $d_{t_0}^{J_2,\alpha}(x,y) \leq  d_{t_0}^{J_1,\alpha}(x,y)$ for $x,y \in X(t_0)$ and the claim follows.
\end{proof}


\begin{definition}[Finite-time metric entropy (FTME)]\label{defFTME}
Let $\phi$ be an FTNDS on $X$ over $J$ and $\alpha \in \R$. The \emph{finite-time metric entropy (FTME)} with weight $\alpha$ at $(t_0,x_0) \in X$ is defined by
\begin{equation}\label{FTentropy}
  h_{t_0}^{\alpha}(x_0)
  :=
  \limsup_{\eps \to 0} h_{t_0}^{\alpha}(x_0,\eps)
  \qquad \text{with }
  h_{t_0}^{\alpha}(x_0,\eps)
  :=
  -\frac{1}{|J|} \log\frac{\mu\big(B_{t_0}^{\alpha}(x_0,\eps)\big)}{\mu\big(B(x_0,\eps)\big)}
  .
\end{equation}
The dependency of $h_{t_0}^{\alpha} = h_{t_0}^{\phi,J,\alpha}$ on $\phi$ and $J$ is sometimes denoted in the superscript.
\end{definition}


\begin{remark}[Finite-time escape rate]\label{remFTME}
(a) Definition \ref{defFTME} can be seen as a finite-time version of the local $\alpha$-entropy introduced by Thieullen \cite{thieullen1} to FTNDS which are not necessarily measure-preserving. However, in contrast to \cite{thieullen1}, we will study FTME for weight factors $\alpha$ which might depend on $x_0$ and are not just a constant. We will exploit this idea in Section 4 in Theorem \ref{autonomentropy} to construct new candidates of Lagrangian coherent structures. 

(b)
The quantity $h_{t_0}^{\alpha}(x_0,\eps)$ in \eqref{FTentropy} is called \emph{finite-time $\alpha$-escape rate} of radius $\eps > 0$ at $(t_0,x_0)$. It measures how many points escape from the $\eps$-orbit neighborhood of the orbit ${\cal O}(t_0,x_0)$ on $J$, since with the first-order approximation $\log x \approx x - 1$ for $x \approx 1$, and using the fact that $B_{t_0}^{\alpha}(x_0,\eps) \subset B(x_0,\eps)$, we have
\[
  h_{t_0}^{\alpha}(x_0,\eps)
  \approx
  \frac{1}{|J|}
  \bigg[
  1 - \frac{\mu\big(B_{t_0}^{\alpha}(x_0,\eps)\big)}{\mu\big(B(x_0,\eps)\big)}
  \bigg]
  =
  \frac{1}{|J|}
  \frac{\mu\big(B(x_0,\eps) \setminus B_{t_0}^{\alpha}(x_0,\eps)\big)}{\mu\big(B(x_0,\eps)\big)}
\]
if $\frac{\mu(B_{t_0}^{\alpha}(x_0,\eps))}{\mu(B(x_0,\eps))} \approx 1$.

(c) Let $\phi$ be an NDS on $\Z \times M$ over $\Z$ generated by a homeomorphism $T : M \rightarrow M$ as in Example \ref{example-NDS}(a). In order to relate the metric entropy $h_\mu(T,x)$ of $T$ at $x$ to the FTME, more precisely, to the finite-time escape rate, define the sets $J_n := \{0, 1, \dots, n\}$ for $n \in \N$. Using the fact that $|J_n| = n$ and $B(x,n,\eps)$ in \eqref{autonomousstableset} equals $B_{t_0}^{J_n,\alpha}(x,\eps)$ for $t_0=0$ and $\alpha=0$, we get $\limsup_{n \to \infty} h_0^{J_n,0}(x,\eps) = \limsup_{n \to \infty} - \frac{1}{n} (\log B_0^{J_n,0}(x,\eps) - \log B(x,\eps))$ and hence
\[
  h_\mu(T,x)
  =
  \limsup_{\eps \to 0}
  \limsup_{n \to \infty} h_0^{J_n,0}(x,\eps)
  .
\]

(d) If $\phi$ is an NDS on $X$ over a two-point set $J = \{t_0, t_0 + T\}$ for some $t_0 \in \R$ and $T > 0$, then  \eqref{stableset} for $\alpha = 0$ implies $B_{t_0}^{0}(x_0,\eps) = B(x_0,\eps) \cap \phi(t_0 + T,t_0)^{-1} B(\phi(t_0+T,t_0)x_0, \eps)$, and with \eqref{FTentropy} we get for the finite-time escape rate $h_{t_0}^{0}(x_0,\eps)$ the relation
\begin{equation}\label{coherentset}
  \frac{\mu\big(\varphi(t_0 +T,t_0)^{-1}B(\varphi(t_0 +T,t_0)x_0,\eps) \cap B(x_0,\eps)\big)}
  {\mu(B(x_0,\eps))}
  =
  e^{- h_{t_0}^{0}(x_0,\eps) T}
  .
\end{equation}
If $h_{t_0}^{0}(x_0,\eps) \approx 0$ then the pair of sets $A_{t_0} := B(x_0,\eps)$, $A_{t_0+T} := B(\varphi(t_0 +T,t_0)x_0,\eps)$, satisfies $A_{t_0} \approx \varphi(t_0 +T,t_0)^{-1} A_{t_0 + T}$ and is called \emph{pair of coherent sets} in \cite{froyland5, froyland2}. In other words, the FTME $h_{t_0}^{0}(x_0)$ over a two-point set $\{t_0, t_0 + T\}$ is an average logarithmic measure of coherence of infinitesimally small balls centered at $x_0$ and $\varphi(t_0 +T,t_0)x_0$.

(e) \emph{Finite-time metric entropy (FTME)} in Definition \ref{defFTME} and \emph{finite-time entropy (FTE)} \cite[Definition 4.1]{froyland} can be expressed in terms of \emph{differential entropy} which is defined by $h_{\operatorname{diff}}(f) = - \int_{\R^n} f(x) \log f(x) \,d\mu(x)$ for $f \in L^1(\R^n)$ and goes back to Boltzmann (see \cite[Chapter 9]{lasota} for a discussion in the dynamical systems context).
FTE for an NDS $\phi$ on $X$ over a two-point time set $J = \{t_0, t_0+T\}$ satisfies
\begin{equation*}
  FTE(x_0,t_0,T)  =
  \lim_{\eps \to  0} \frac{1}{|T|} \big[ h_{\operatorname{diff}}(\mathcal{A}_\eps \mathcal{P}_{t_0,T} f_{B(x_0,\eps)})
  - h_{\operatorname{diff}}(f_{B(x_0,\eps)}) \big]
  \quad \text{with }
  f_A := \frac{1}{\mu(A)} \mathds{1}_A(\cdot)
\end{equation*}
and compares the differential entropy of a scaled characteristic function on an $\eps$-ball with a push-forward of that function by the Perron-Frobenius operator $\mathcal{P}_{t_0,T} f(x) = \frac{f(\varphi(t_0+T,t_0)^{-1}x)}{|\det D_x \phi(t_0+T,t_0) \varphi(t_0+T,t_0)^{-1}x|}$ followed by an $\eps$-smoothing $\mathcal{A}_{\eps} f(x) = \frac{1}{\mu(B(x,\eps))} \int_{B(x,\eps)} f \,d\mu$.

FTME for an NDS $\phi$ on $X$ over a compact time set $J \subset \R$ for some $t_0 \in J$ and weight factor $\alpha \in \R$ is
\begin{equation*}
  h_{t_0}^{\alpha}(x_0)
  = \lim \limits_{\eps \to 0} \frac{1}{|J|}
  \big[ h_{\operatorname{diff}}(f_{B_{t_0}^{\alpha}(x_0,\eps)}) - h_{\operatorname{diff}}(f_{B(x_0,\eps)}) \big]
  .
\end{equation*}
The comparison of FTME and FTE will be the subject of further studies. To illustrate one possible relation between FTE and FTME, let $\phi$ be an NDS on $X$ over a two-point time set $J = \{t_0, t_0+T\}$. Assume for simplicity that $X(t_0) = X(t_0 + T)$ and let ${\cal B} = \{B_1, \dots, B_n\}$ be a partition of the state space $X(t_0) = X(t_0+T)$. Then formula \eqref{coherentset} suggests that the FTME $h_{t_0}^{0}(x_0)$ could be approximated by
\[
  - \frac{1}{T} \log \frac{\mu(\phi(t_0 + T,t_0)^{-1} B_j \cap B_i)}{\mu(B_i)}
\]
where $x_0 \in B_i$ and $\phi(t_0+T,t_0)x_0 \in B_j$ for some $i,j$. On the other hand, $FTE(x_0,t_0,T)$, with $x_0 \in B_i$ for some $i$, is approximated by a localized version of the Kolmogorov-Shannon entropy
\[
  -\frac{1}{T} \sum_{j=1}^n \frac{\mu(B_i \cap \phi(t_0 + T,t_0)^{-1}B_j)}{\mu(B_i)}
  \log \frac{\mu(B_i \cap \phi(t_0 + T,t_0)^{-1}B_j)}{\mu(B_i)}
\]
of the partition $\cal B$.
\end{remark}


The following proposition states that FTME is constant for linear nonautonomous dynamical systems. A similar statement for FTE can be found in \cite[Lemma 2.6]{froyland}. Note, however, that the FTME with an exponential weight factor $\alpha = \alpha(x_0)$ which depends on $x_0 \in X(t_0)$ for some $t_0 \in J$, in general is not constant even for linear systems. Indeed the weighted FTME field is able to detect stable and unstable manifolds (see Example \ref{exlinear}).

\begin{proposition}[Linearity implies constant FTME]\label{FTEcompute-linear}
Let $\phi$ be an FTNDS on $X$ over $J$ and $\alpha \in \R$. Assume that $\phi(t,s) : X(s) \to X(t)$ is linear for all $t,s \in J$. Then $h_{t_0}^{\alpha}(x_0)$ is independent of $x_0$ and is denoted by $h_{t_0}^{\alpha}$ which satisfies
\begin{equation}\label{entropylinear}
h^\alpha_{t_0} = -\frac{1}{|J|} \log \frac{\mu(B^\alpha_{t_0}(0,1))}{\mu(B(0,1))}.
\end{equation}
\end{proposition}

\begin{proof}
Since $\phi$ is linear
\begin{eqnarray*}
  B_{t_0}^{\alpha}(x_0,\eps)
  &=&
  \big\{x \in X(t_0) : \sup_{t \in J} \|\phi(t,t_0)(x-x_0)\| e^{-\alpha (t-t_0)} \leq \eps\big\}
\\
  &=&
  \big\{x \in X(t_0) : x-x_0 \in B_{t_0}^{\alpha}(0,\eps)\big\}
  =
  x_0 + B_{t_0}^{\alpha}(0,\eps)\\
  &=& x_0 + \bigcap_{t\in J} \phi(t,t_0)^{-1}B(0,\eps e^{\alpha(t-t_0)}) \\ 
  &=& x_0 + \eps \bigcap_{t\in J} \phi(t,t_0)^{-1}B(0,e^{\alpha(t-t_0)}) = x_0 + \eps B^\alpha_{t_0}(0,1).
\end{eqnarray*}
Since $\mu$ is the $n$-dimensional Lebesgue measure which is translation invariant, it follows that
\[
  \mu\big(B_{t_0}^{\alpha}(x_0,\eps)\big)
  =
  \mu\big(B_{t_0}^{\alpha}(0,\eps)) = \eps^n \mu\big(B_{t_0}^{\alpha}(0,1))
\]
and
\[
  \mu\big(B(x_0,\eps)\big) = \mu\big(B(0,\eps)\big) = \eps^n  \mu\big(B(0,1)\big)
  ,
\]
proving that the FTME is independent of $x_0$ and is given by \eqref{entropylinear}.
\end{proof}

The following theorem shows that the weighted FTME of a nonlinear nonautonomous dynamical system equals the weighted FTME of its linearization. A similar statement also holds for FTE \cite[Lemma 2.7]{froyland}.

\begin{theorem}[Linearized FTME]\label{FTEcompute}
Let $\phi$ be a $C^2$ FTNDS on $X$ over $J$ and $\alpha \in \R$.
Then the linearization $\Phi_{x_0} (t,s) := D \phi(t,s)x_0$ determines the FTME, and for $(t_0,x_0) \in X$
\begin{equation}\label{FTentropy2}
  h^{\phi, \alpha}_{t_0}(x_0) = h^{\Phi_{x_0},\alpha}_{t_0}.
\end{equation}
\end{theorem}

\begin{proof}
Let $(t_0,x_0) \in X$. Then Taylor's formula implies for $t \in J$ and $x \in X(t_0)$
\begin{equation}\label{eq-taylor-phi}
  \phi(t,t_0)x - \phi(t,t_0)x_0 = \Phi_{x_0}(t,t_0) (x-x_0) + r(t,t_0,x-x_0)
\end{equation}
with a continuous function $r$ which satisfies $\lim_{x \to x_0} \frac{r(t,t_0,x-x_0)}{\|x - x_0\|^2} = 0$ uniformly in $t, t_0 \in J$. Choose and fix $\eps_0 > 0$. Then there exists a constant $\tilde{C} > 0$ such that for all $\eps \in [0,\eps_0]$ and $t,s \in J$, $\|z\| \leq \eps$
\begin{equation}\label{eq-taylor-remainder}
  \|r(t,s,z)\| \leq \tilde{C} \eps^2
  .
\end{equation}
We show the following two inclusions for $\eps \in [0,\eps_0]$
\[
  \mathrm{(i)}\, B_{t_0}^{\phi,\alpha}(x_0,\eps) \subseteq B_{t_0}^{\Phi_{x_0},\alpha}(x_0,\eps + C\eps^2)
  \qquad \text{and} \qquad
  \mathrm{(ii)}\, B_{t_0}^{\Phi_{x_0},\alpha}(x_0,\eps) \subseteq B_{t_0}^{\phi,\alpha}(x_0,\eps + C\eps^2)
\]
with $C := \tilde{C} \sup_{t \in J} e^{-\alpha(t-t_0)}$. To show (i), let $x \in B_{t_0}^{\phi,\alpha}(x_0,\eps)$. Then $\|\phi(t,t_0)x - \phi(t,t_0)x_0\| \leq \eps e^{\alpha(t-t_0)}$ for all $t \in J$. With \eqref{eq-taylor-phi} and \eqref{eq-taylor-remainder} we get for $t \in J$
\[
  \|\Phi_{x_0}(t,t_0)(x - x_0)\| e^{-\alpha(t-t_0)} \leq \eps + \tilde{C}\eps^2 e^{-\alpha(t-t_0)}
\]
and taking the supremum over $t \in J$ yields (i). The inclusion (ii) is proved analogously.
Applying the Lebesgue measure $\mu$ to (i), we get
\begin{equation}\label{star}
  \frac{\mu\big(B_{t_0}^{\phi,\alpha}(x_0,\eps)\big)}{\mu\big(B(x_0,\eps)\big)}
  \frac{\mu\big(B(x_0,\eps)\big)}{\mu\big(B(x_0,\eps + C\eps^2)\big)}
  \leq
  \frac{\mu\big(B_{t_0}^{\Phi_{x_0},\alpha}(x_0,\eps + C\eps^2)\big)}{\mu\big(B(x_0,\eps + C\eps^2)\big)}
  .
\end{equation}
Taking the logarithm, dividing by $|J|$, letting $\eps \to 0$ and using the fact that $\tfrac{\mu(B(x_0,\eps))}{\mu(B(x_0,\eps + C\eps^2))} = \frac{\eps^n}{(\eps + C\eps^2)^n} \to 1$, we get $h^{\phi, \alpha}_{t_0}(x_0) \leq h^{\Phi_{x_0},\alpha}_{t_0}$. Similarly (ii) implies $h^{\Phi_{x_0},\alpha}_{t_0} \leq h^{\phi, \alpha}_{t_0}(x_0)$, proving \eqref{FTentropy2}.
\end{proof}


\begin{remark}\label{rem:9}
(a) From Theorem \ref{FTEcompute} and its proof one can derive that for $C^2$ FTNDS the limsup in Definition \ref{defFTME} of FTME can be replaced by lim.

(b) If the Euclidean norm in \eqref{weightedmetric} is replaced by a norm $\|\cdot\|_\Gamma:= \|\Gamma \cdot\|$ for a positive definite matrix $\Gamma \in \R^{n\times n}$, then the finite-time metric entropy w.r.t.\ the $\|\cdot\|_\Gamma$ norm is defined by
\begin{equation}
  h_{t_0}^{\Gamma,\alpha}(x_0)
  :=
  \limsup_{\eps \to 0} h_{t_0}^{\Gamma,\alpha}(x_0,\eps)
  \qquad \text{with }
  h_{t_0}^{\Gamma,\alpha}(x_0,\eps)
  :=
  -\frac{1}{|J|} \log\frac{\mu\big(B_{t_0}^{\Gamma,\alpha}(x_0,\eps)\big)}{\mu\big(B^\Gamma(x_0,\eps)\big)}
  .
\end{equation}
with $B_{t_0}^{\Gamma,\alpha}(x_0,\eps) := \{x \in X(t_0): \sup_{t\in J} \|\phi(t,t_0)y- \phi(t,t_0)x_0\|_\Gamma e^{-\alpha(t-t_0)} \leq \eps \}$ and $B^\Gamma(x_0,\eps):= \{x \in X(t_0): \|x-x_0\|_\Gamma \leq \eps \}$.
Similarly as in the proofs of Proposition \ref{FTEcompute-linear} and Theorem \ref{FTEcompute}, one can show that $h_{t_0}^{\Gamma,\alpha}(x_0)$ equals the FTME of the linearization at $x_0$, which is a constant.   
\end{remark}


To geometrically characterize FTME using ellipsoids, recall that for an invertible matrix $A \in \R^{n \times n}$ the ellipsoid
\[
  E(A) := A^{-1} B(0,1)
  =
  \big\{ A^{-1} x \in \R^n : x \in B(0,1) \big\}
  = \big\{ x \in \R^n : \langle x, A^\top A x \rangle \leq 1 \big\}
\]
is the unit ball in the new norm $\|\cdot\|_{A^\top A} = \sqrt{\langle x, A^\top A x \rangle}$ induced by the symmetric positive definite matrix $A^\top A = U \Lambda^2 U^T$ where $V \Lambda U^T = A$ is the singular value decomposition of $A$ with orthogonal matrices $U$, $V$ and diagonal matrix $\Lambda = \diag(\Lambda_1, \dots, \Lambda_n)$ with singular values $\Lambda_1 \geq \dots \geq \Lambda_n > 0$. The semi-principal axes of $E(A)$ are described by the $n$ unit vectors which form the columns of $U$ and have length $\Lambda_i^{-1}$, $i = 1, \dots, n$.


\begin{corollary}[Ellipsoid characterization of entropy]
\label{thm-ellipsoid}
Under the assumptions of Theorem \ref{FTEcompute}, for $(t_0,x_0) \in X$
\begin{equation}\label{FTElinear2}
  h_{t_0}^{\phi,\alpha}(x_0)
  =
  - \frac{1}{|J|} \log \frac{\Gamma(\frac{n}{2} + 1)}{\pi^{\frac{n}{2}}}
  \mu \Big( \bigcap_{t \in J} E\big(\Phi_{x_0}(t,t_0) e^{-\alpha(t-t_0)}\big) \Big)
  .
\end{equation}
\end{corollary}

\begin{proof}
Using the facts that
\begin{eqnarray*}
  B_{t_0}^{\Phi_{x_0}, \alpha}(0,1)
  & = &
  \{ x \in \R^n : \|\Phi_{x_0}(t,t_0)x\| \leq e^{\alpha(t-t_0)} \text{ for } t \in J\}
\\
  & = &
  \{ x \in \R^n : \langle \Phi_{x_0}(t,t_0) e^{-\alpha(t-t_0)}x, \Phi_{x_0}(t,t_0) e^{-\alpha(t-t_0)}x \rangle \leq 1
   \text{ for } t \in J\}
\\
  & = &
  \bigcap_{t \in J} E(\Phi_{x_0}(t,t_0) e^{-\alpha(t-t_0)})
\end{eqnarray*}
and $B(0,1) = \frac{\pi^{\frac{n}{2}}}{\alpha(\frac{n}{2} + 1)}$, we get with Theorem \ref{FTEcompute} for $(t_0,x_0) \in X$
\[
  h^{\phi, \alpha}_{t_0}(x_0)
  =
  h^{\Phi_{x_0},\alpha}_{t_0}
  =
  - \frac{1}{|J|} \log \frac{\mu\big(B_{t_0}^{\Phi_{x_0}, \alpha}(0,1)\big)}{\mu\big(B(0,1)\big)}
  ,
\]
proving \eqref{FTElinear2}.
\end{proof}





\begin{theorem}\label{thm1}
Under the assumptions of Theorem \ref{FTEcompute} the following holds.
\begin{itemize}
  \item[(i)] \emph{Upper and lower bound on FTME:} For $(t_0,x_0) \in X$
    \begin{equation}\label{prop2}
      \frac{n}{|J|} \log \Big(\sup_{t\in J} e^{\alpha(t-t_0)}\|\Phi_{x_0}(t,t_0)^{-1}\|\Big)
      \leq
      h^{\phi,\alpha}_{t_0}(x_0)
      \leq
      \frac{n}{|J|} \log \Big(\inf_{t\in J}\frac{e^{\alpha(t-t_0)}}{\|\Phi_{x_0}(t,t_0)\|}\Big)
  \end{equation}

  \item[(ii)] \emph{FTME along trajectories:} For $t \in J$ and $(t_0,x_0) \in X$
    \begin{equation}
       h^{\phi,\alpha}_{t}(\varphi(t,t_0)x_0)
      =
      h^{\phi,\alpha}_{t_0}(x_0)
      + \frac{n \alpha(t-t_0)}{|J|} - \frac{\log |\det \Phi_{x_0}(t,t_0)|}{|J|} \label{eq02}
    \end{equation}
\end{itemize}
\end{theorem}

\begin{proof}
(i) Let $(t_0,x_0) \in X$. By Theorem \ref{FTEcompute}, $h^{\phi,\alpha}_{t_0}(x_0) = h^{\Phi_{x_0},\alpha}_{t_0}$.
To prove \eqref{prop2} with $h^{\phi,\alpha}_{t_0}(x_0)$ replaced by $h^{\Phi_{x_0},\alpha}_{t_0}$, we first prove for $(t_0,x_0) \in X$ that
\begin{equation}\label{prop1}
  B\Big(0,\inf_{t\in J}\frac{e^{\alpha(t-t_0)}}{\|\Phi_{x_0}(t,t_0)\|}\Big)
  \subset
  B^{\alpha}_{t_0}(0,1)
  \subset
  B\Big(0,\sup_{t\in J} e^{\alpha(t-t_0)}\|\Phi_{x_0}(t,t_0)^{-1}\|\Big)
\end{equation}
where $B^{\alpha}_{t_0}(0,1) = B_{t_0}^{\Phi_{x_0}, \alpha}(0,1) = \{ x \in \R^n : \|\Phi_{x_0}(t,t_0)x\| \leq e^{\alpha(t-t_0)} \text{ for all } t \in J\}$.
Let $y \in B^{\alpha}_{t_0}(0,1)$. Then for $t \in J$
\[
  \frac{\|y\|}{\|\Phi_{x_0}(t,t_0)^{-1}\|}
  \leq
  \|\Phi_{x_0}(t,t_0)y\|
  \leq
  e^{\alpha(t-t_0)},
\]
hence $\|y\| \leq e^{\alpha(t-t_0)}\|\Phi_{x_0}(t,t_0)^{-1}\|$ and therefore $y \in B(0,\sup_{t\in J} e^{\alpha(t-t_0)}\|\Phi_{x_0}(t,t_0)^{-1}\|)$.
Let $y \in B(0,\inf_{t\in J}\frac{e^{\alpha(t-t_0)}}{\|\Phi_{x_0}(t,t_0)\|})$. Then $\|y\| \leq \inf_{t\in J} \frac{e^{\alpha(t-t_0)}}{\|\Phi_{x_0}(t,t_0)\|}$ for $t \in J$ and hence
\[
  \|\Phi_{x_0}(t,t_0)y\| \leq \|\Phi_{x_0}(t,t_0)\| \|y\| \leq e^{\alpha(t-t_0)},
\]
proving that $y \in B^{\alpha}_{t_0}(0,1)$ and thus \eqref{prop1}. Property \eqref{prop2} then follows by taking the negative logarithms of the measures of the sets in \eqref{prop1} divided by $\mu(B(0,1)) = \frac{\pi^{\frac{n}{2}}}{\alpha(\frac{n}{2} + 1)}$.

(ii) To prove \eqref{eq02}, observe that since $\mu$ is the Lebesgue measure,
\begin{eqnarray}
  & & h^{\varphi,\alpha}_{t}(\varphi(t,t_0)x_0) \nonumber
\\
  &=&
  \nonumber
  -\lim_{\eps \to 0} \frac{1}{|J|}
  \log \frac{\mu \big(B_{t}^{\varphi,\alpha} (\varphi(t,t_0)x_0,\eps)\big)}{\mu\big(B(\varphi(t,t_0)x_0,\eps)\big)}
\\ \nonumber
  &=&
  -\lim_{\eps \to 0} \frac{1}{|J|}
  \log \left( \frac{\mu \big(\varphi(t,t_0) B_{t_0}^{\varphi,\alpha} (x_0,\eps e^{-\alpha(t-t_0)})\big)}
  {\mu\big(B(x_0,\eps e^{-\alpha(t-t_0)})\big)}
  \frac{\mu\big(B(x_0,\eps e^{-\alpha (t-t_0)})\big)}{\mu\big(B(\varphi(t,t_0)x_0,\eps)\big) } \right)
\\ \nonumber
  &=&
  -\lim_{\eps \to 0} \frac{1}{|J|}
  \log \left(\frac{\mu \big(\varphi(t,t_0) B_{t_0}^{\varphi,\alpha} (x_0,\eps e^{-\alpha(t-t_0)})\big)}
  {\mu\big(B(x_0,\eps e^{-\alpha(t-t_0)})\big)} e^{-n\alpha(t-t_0)}\right)
\\ \nonumber
  &=&
  - \lim_{\eps \to 0} \frac{1}{|J|}
  \log \left(\frac{\mu \big(\varphi(t,t_0) B_{t_0}^{\varphi,\alpha} (x_0,\eps e^{-\alpha(t-t_0)})\big)}
  {\mu\big(B_{t_0}^{\varphi,\alpha} (x_0,\eps e^{-\alpha(t-t_0)})\big)}
  \frac{\mu\big(B_{t_0}^{\varphi,\alpha} (x_0,\eps e^{-\alpha(t-t_0)})\big)}
  {\mu\big(B(x_0,\eps e^{-\alpha(t-t_0)})\big)}
  e^{-n\alpha(t-t_0)}\right)
\\
  &=& h_{t_0}^{\phi,\alpha}(x_0) +
  \frac{n \alpha(t-t_0)}{|J|} -
  \lim_{\eps \to 0}
  \frac{1}{|J|}
  \log \frac{\mu \big(\varphi(t,t_0) B_{t_0}^{\varphi,\alpha} (x_0,\eps e^{-\alpha(t-t_0)})\big)}
  {\mu\big(B_{t_0}^{\varphi,\alpha} (x_0,\eps e^{-\alpha(t-t_0)})\big)},
  \label{eqinvariant1}
\end{eqnarray}
in case the limit in the last line of \eqref{eqinvariant1} exists. Using the abbreviation $\bar{B}= B_{t_0}^{\alpha}(x_0,\eps e^{-\alpha(t-t_0)})$, we apply \cite[Theorem H.1]{taylor} to get
\begin{eqnarray*}
  \mu(\varphi(t,t_0)\bar{B})
  &=&
  \int_{\varphi(t,t_0)\bar{B}} 1_{\varphi(t,t_0)\bar{B}}(y)d\mu(y)
  =
  \int_{\bar{B}} 1_{\varphi(t,t_0)\bar{B}}(\varphi(t,t_0)x) |\det \Phi_x(t,t_0)| d\mu(x)
\\
  &=&
  \int_{\bar{B}} 1_{\bar{B}}(x) |\det \Phi_x(t,t_0)| d\mu(x).
\end{eqnarray*}
Hence
\begin{eqnarray*}
  & &
  \big| \mu(\varphi(t,t_0)\bar{B}) - |\det \Phi_{x_0}(t,t_0)| \mu(\bar{B}) \big|
\\
  & = &
  \Big| \int_{\bar{B}} 1_{\bar{B}}(x) |\det \Phi_x(t,t_0) | d\mu(x) -
  \int_{\bar{B}} 1_{\bar{B}}(x) |\det \Phi_{x_0}(t,t_0)| d\mu(x) \Big|
\\
  & \leq &
  \int_{\bar{B}} 1_{\bar{B}}(x) \sup_{x \in \bar{B}}
  \Big| |\det \Phi_{x}(t,t_0)| - |\det \Phi_{x_0}(t,t_0)| \Big| d\mu(x)
\\
  & \leq & \mu(\bar{B}) \sup_{x \in \bar{B}} \Big| |\det \Phi_{x}(t,t_0)| - |\det \Phi_{x_0}(t,t_0)| \Big|
\\
  & \leq & \mu(\bar{B}) \sup_{x \in B(x_0,\eps e^{-\alpha(t-t_0)})} \Big|
  |\det \Phi_{x}(t,t_0)| - |\det \Phi_{x_0}(t,t_0)| \Big|,
\end{eqnarray*}
where the last estimate follows from the inclusion $\bar{B} \subset B(x_0,\eps e^{-\alpha(t-t_0)})$. Due to the continuity of $|\det \Phi_x(t,t_0)|$ at $x_0$, the supremum in the last line of the above chain of inequalities tends to $0$ as $\eps \to 0$. Thus it follows that
\begin{equation}\label{eqinvariant2}
  \lim_{\eps \to 0} \frac{\mu(\varphi(t,t_0)\bar{B})}{\mu(\bar{B})}
  =
  |\det \Phi_{x_0}(t,t_0)|
\end{equation}
and \eqref{eq02} is a consequence of \eqref{eqinvariant1} and \eqref{eqinvariant2}.
\end{proof}


The following theorem estimates the change of the FTME $h^{\alpha}_{t_0}(x_0)$ under a change from the Euclidean norm $\|\cdot\|$ to a new norm $\|\cdot\|_{\Gamma} := \|\Gamma \cdot\|$ in $\R^n$ with a positive definite matrix $\Gamma \in \R^{n\times n}$ (see also Remark \ref{rem:9}(b)).  


\begin{theorem}\label{FTMEnorm}
Let $\phi$ be a $C^2$ FTNDS on $X$ over $J$ and $\alpha \in \R$. Then the following estimate holds
\begin{equation}\label{normes}
  |h^{\Gamma,\alpha}_{t_0}(x_0) - h^{\alpha}_{t_0}(x_0)| 
  \leq 
  \frac{n \log \|\Gamma\|+ n \log \|\Gamma^{-1}\| }{|J|}.
\end{equation}
\end{theorem}


\begin{proof}
Under the new norm $\|\cdot\|_\Gamma = \|\Gamma \cdot\|$, the corresponding fibre metric $d^{\Gamma,\alpha}_{t_0}(y,x)$ becomes
\[
d^{\Gamma,\alpha}_{t_0}(y,x) = \sup_{t\in J} \|\Gamma (\varphi(t,t_0)y - \varphi(t,t_0)x)\| e^{-\alpha(t-t_0)}
\]
It is easy to see that for $t, t_0 \in J$, $x, y \in X(t_0)$
\begin{eqnarray*}
\frac{1}{\|\Gamma^{-1}\|} \|(\varphi(t,t_0)y - \varphi(t,t_0)x)\| e^{-\alpha(t-t_0)}
&\leq& \|\Gamma (\varphi(t,t_0)y - \varphi(t,t_0)x)\| e^{-\alpha(t-t_0)} \\
&\leq& \|\Gamma\| \|(\varphi(t,t_0)y - \varphi(t,t_0)x)\| e^{-\alpha(t-t_0)},
\end{eqnarray*}
hence
\begin{eqnarray*}
\sup_{t\in J} \frac{1}{\|\Gamma^{-1}\|} \|(\varphi(t,t_0)y - \varphi(t,t_0)x)\| e^{-\alpha(t-t_0)}&\leq& \sup_{t\in J} \|\Gamma (\varphi(t,t_0)y - \varphi(t,t_0)x)\| e^{-\alpha(t-t_0)} \\
&\leq& \sup_{t\in J} \|\Gamma\| \|(\varphi(t,t_0)y - \varphi(t,t_0)x)\| e^{-\alpha(t-t_0)},
\end{eqnarray*}
which proves that for $t_0 \in J$, $x, y \in X(t_0)$
\[
\frac{1}{\|\Gamma^{-1}\|} d^{\alpha}_{t_0}(y,x) \leq d^{\Gamma,\alpha}_{t_0}(y,x) \leq \|\Gamma\| d^{\alpha}_{t_0}(y,x).
\]
As a consequence, for each $t_0 \in J$, $x_0 \in X(t_0)$ and $\eps >0$, we have the inclusions
\[
B_{t_0}^{\alpha}(x_0,\frac{\eps}{\|\Gamma\|}) \subset B_{t_0}^{\Gamma,\alpha}(x_0,\eps) \subset B_{t_0}^{\alpha}(x_0,\eps\|\Gamma^{-1}\|).
\]
On the other hand, we also have
\[
B(x_0,\frac{\eps}{\|\Gamma\|}) \subset B^{\Gamma}(x_0,\eps) \subset B(x_0,\eps \|\Gamma^{-1}\|).
\]
Therefore
\begin{equation*}
  \frac{\mu(B_{t_0}^{\alpha}(x_0,\frac{\eps}{\|\Gamma\|}))}{\mu(B(x_0,\eps \|\Gamma^{-1}\|))} 
  \leq
  \frac{\mu(B_{t_0}^{\Gamma,\alpha}(x_0,\eps))}{\mu(B^{\Gamma}(x_0,\eps))} 
  \leq 
  \frac{\mu(B_{t_0}^{\alpha}(x_0,\eps \|\Gamma^{-1}\|))}{\mu(B(x_0,\frac{\eps}{\|\Gamma\|}))}
  ,
\end{equation*}
and using the fact that $\mu$ is the $n$-dimensional Lebesgue measure, we have
\begin{equation*}
  \frac{1}{\|\Gamma\|^n\|\Gamma^{-1}\|^n} 
  \frac{\mu(B_{t_0}^{\alpha}(x_0,\frac{\eps}{\|\Gamma\|}))}{\mu(B(x_0,\frac{\eps}{ \|\Gamma\|}))} 
  \leq 
  \frac{\mu(B_{t_0}^{\Gamma,\alpha}(x_0,\eps))}{\mu(B^{\Gamma}(x_0,\eps))} 
  \leq 
  \|\Gamma\|^n\|\Gamma^{-1}\|^n 
  \frac{\mu(B_{t_0}^{\alpha}(x_0,\eps \|\Gamma^{-1}\|))}{\mu(B(x_0,\eps\|\Gamma^{-1}\|))}
  .
\end{equation*}
Taking the limit as $\eps \to 0$ and using Definition \ref{defFTME}, we get
\[
  - \frac{n \log \|\Gamma\| + n \log \|\Gamma^{-1}\|}{T} + h^{\alpha}_{t_0}(x_0) \leq h^{\Gamma,\alpha}_{t_0}(x_0) 
  \leq
  \frac{n \log \|\Gamma\| + n \log \|\Gamma^{-1}\|}{T} + h^{\alpha}_{t_0}(x_0),
\]
which then implies \eqref{normes}.
\end{proof}


\section{Pesin's formula}

Pesin's formula in Theorem \ref{alphaentropy} relates local entropy to the sum of Lyapunov exponents which are not less than the weight factor $\alpha$. We prove a finite-time version and relate the FTME to the sum of positive finite-time Lyapunov exponents. Let $\phi$ be a $C^1$ NDS on $X$ over a two-point set $J = \{t_0, t_0+T\}$ for some $t_0 \in \R$ and $T>0$. Let $\Lambda_i(t_0,x_0,T)$ denote the singular values of $D_{x_0}\phi(t_0+T,t_0) := D\phi(t_0+T,t_0)x_0$, i.e.\ $D_{x_0}\phi(t_0+T,t_0)^\top D_{x_0}\phi(t_0+T,t_0) = U \Lambda^2 U^\top$ with $\Lambda := \diag(\Lambda_1(t_0,x_0,T), \dots, \Lambda_n(t_0,x_0,T))$ and an orthogonal matrix $U$. The \emph{finite-time Lyapunov exponents (FTLE)} or \emph{time-$T$ Lyapunov exponents} $\lambda_i(t_0,x_0,T)$ of $\phi$ at $(t_0,x_0)$ are defined by $\Lambda_i(t_0,x_0,T) = e^{\lambda_i(t_0,x_0,T)T}$, or explicitly
\[
  \lambda_i(t_0,x_0,T) := \frac{1}{T} \log \Lambda_i(t_0,x_0,T),
  \qquad i = 1,\dots,n
  .
\]
In order to relate the FTLE to the FTME, we use formula \eqref{FTElinear2} in Corollary \ref{thm-ellipsoid}. The fact that the ellipse $E(I_{n \times n})$ of the identity matrix $I_{n \times n}$ equals $B(0,1)$ then implies
\begin{equation}\label{equ-ellipsoid}
  h_{t_0}^{\alpha}(x_0)
  =
  - \frac{1}{T} \log \frac{\Gamma(\frac{n}{2} + 1)}{\pi^{\frac{n}{2}}}
  \mu (M)
  \quad \text{with }
  M := B(0,1) \cap E\big(D_{x_0}\phi(t_0+T,t_0) e^{-\alpha T}\big)
  .
\end{equation}
For a scalar NDS $\phi : \{t_0, t_0+T\}^2 \times \R \rightarrow \R$ a direct computation shows that for $\alpha \in \R$ the following scalar finite-time version of Pesin's formula holds
\begin{equation}\label{entropy1d}
  h_{t_0}^{\alpha}(x_0)
  =
  \big(\lambda_1(t_0,x_0,T) - \alpha\big)^+
  .
\end{equation}
Using the ellipsoidal representation \eqref{equ-ellipsoid} of FTME, one could in principle compute $h_{t_0}^{\alpha}(x_0)$ explicitly and also its relation to the FTLE, deriving an exact finite-time Pesin's formula. However, it turns out that the computation and formula is very complicated even for three-dimensional systems. The following proposition provides an explicit formula for $n=2$.

\begin{proposition}[Exact Pesin's formula for two-dimensional FTNDS]\label{Pesin2D}
Let $\phi$ be a $C^2$ NDS on $X \subset J \times \R^2$ over a two-point interval $J = \{t_0,t_0+T\}$ for some $t_0\in \R$ and $T>0$ and let $\alpha \in \R$. Then for $x_0 \in X(t_0)$
\begin{eqnarray}\label{entropy2d}
  h_{t_0}^{\alpha}(x_0)
  =
  \begin{cases}
  0 &  \text{if } 1 \geq \kappa_1 \geq \kappa_2
  ,
\\
  - \frac{1}{T} \log \frac{2}{\pi} \left( \arccos \sqrt{\frac{\kappa_1^2-1}{\kappa_1^2 - \kappa_2^2}} +
  \frac{1}{\kappa_1 \kappa_2} \arccos \kappa_1 \sqrt{\frac{1-\kappa_2^2}{\kappa_1^2 - \kappa_2^2}}\right) &
  \text{if } \kappa_1 > 1 > \kappa_2
  ,
\\
  - \frac{1}{T} \log \frac{1}{\kappa_1 \kappa_2} & \text{if } \kappa_1 \geq \kappa_2 \geq 1
  ,
  \end{cases}
\end{eqnarray}
where $\kappa_1 := e^{(\lambda_1(t_0,x_0,T)-\alpha) T} \geq \kappa_2 := e^{(\lambda_2(t_0,x_0,T)-\alpha) T} >0$. \end{proposition}


\begin{proof}
Using the fact that the intersection $M := B(0,1) \cap E(D_{x_0}\phi(t_0+T,t_0) e^{-\alpha T})$ satisfies $M = B(0,1)$ if $1 \geq \kappa_1 \geq \kappa_2$ and $M = E(D_{x_0}\phi(t_0+T,t_0) e^{-\alpha T})$ if $\kappa_1 \geq \kappa_2 \geq 1$, the claim follows from \eqref{equ-ellipsoid}. In order to compute $\mu(M)$ in case $\kappa_1 > 1 > \kappa_2$, note that $B(0,1)$ intersects the ellipsoid $E(D_{x_0}\phi(t_0+T,t_0) e^{-\alpha T})$ at four points in the plane
\[
  (a_1,b_1), (-a_1,b_1), (a_1,-b_1), (-a_1,-b_1)
  \qquad \text{where }
  a_1 = \frac{1-\kappa_1^2}{\kappa_1^2 - \kappa_2^2}
  \text{ and }
  b_1 = \frac{\kappa_1^2-1}{\kappa_1^2-\kappa_2^2}<1
  .
\]
Thus
\begin{eqnarray*}
  \mu(M)
  & = &
  4 \int_0^{\kappa_1^{-1}} y \,dx
  =
  4 \int_0^{a_1} y \,dx + 4 \int_{a_1}^{\kappa_1^{-1}} y \,dx
\\
  & = &
  4 \int_0^{a_1} \sqrt{1-x^2} \,dx + 4 \int_{a_1}^{\kappa_1^{-1}} \frac{1}{\kappa_2}\sqrt{1-\kappa_1^2 x^2}\,dx
\\
  & = &
  2 \arccos \sqrt{\frac{\kappa_1^2-1}{\kappa_1^2 - \kappa_2^2}} +
  \frac{2}{\kappa_1 \kappa_2} \arccos \kappa_1 \sqrt{\frac{1-\kappa_2^2}{\kappa_1^2 - \kappa_2^2}}
  ,
\end{eqnarray*}
proving \eqref{entropy2d}.
\end{proof}


\begin{corollary}\label{cor.incom}[Pesins's formula for two-dimensional incompressible FTNDS]
Under the assumptions of Proposition \eqref{Pesin2D}, and if $\lambda_1(t_0,x_0,T) + \lambda_2(t_0,x_0,T) = 0$ for  $x_0 \in X(t_0)$, then for $\alpha = 0$
\begin{equation*}
  h_{t_0}^{0}(x_0)
  =
  - \frac{1}{T} \log\bigg(  \frac{4}{\pi} \arccos
  \sqrt{\frac{e^{2\lambda_1(t_0,x_0,T) T}}{e^{2\lambda_1(t_0,x_0,T)T}+1}} \bigg)
  .
\end{equation*}
\end{corollary}


\begin{remark}\label{FTMEvsFTLE}
Note that $h_{t_0}^{0}(x_0)$ in Corollary \ref{cor.incom} can be written as the composition $h_{t_0}^{0}(x_0) = g(\lambda_1(t_0,x_0,T))$ with the strictly monotonically increasing function $g(\lambda) =   - \frac{1}{T} \log\big(  \frac{4}{\pi} \arccos \sqrt{\frac{e^{2\lambda T}}{e^{2\lambda T}+1}} \big)$. Hence the ridge and trough-like structures of the FTME field $x_0 \mapsto h_{t_0}^{0}(x_0)$ and the FTLE field $x_0 \mapsto \lambda_1(t_0,x_0,T)$ coincide, and a (weak) LCS in the sense of Theorem \ref{thmFTLE} could also be defined utilizing the FTME field instead of the FTLE field.
\end{remark}


The following theorem is a local and finite-time version of Pesin's entropy formula.

\begin{theorem}[Finite-time Pesin's formula]\label{thm2}
Let $\phi$ be a $C^2$ NDS on $X$ over a two-point interval $J = \{t_0,t_0+T\}$ for some $t_0\in \R$ and $T>0$ and let $\alpha \in \R$. Then for $x_0 \in X(t_0)$
\begin{equation}\label{approx}
  0
  \leq
  \sum_{i=1}^n \big(\lambda_i(t_0,x_0,T) - \alpha \big)^+
  -
  h_{t_0}^{\alpha}(x_0)
  \leq
  \frac{n \log 2 + \log \alpha(\frac{n}{2}+1) - \frac{n}{2} \log \pi}{T}
\end{equation}
where $a^+ = \max\{a,0\}$.
\end{theorem}


\begin{proof}
Using the fact that the ellipse $E(I_{n \times n})$ of the identity matrix $I_{n \times n}$ equals $B(0,1)$, formula \eqref{FTElinear2} in Corollary \ref{thm-ellipsoid} implies
\[
  h_{t_0}^{\alpha}(x_0)
  =
  - \frac{1}{T} \log \frac{\Gamma(\frac{n}{2} + 1)}{\pi^{\frac{n}{2}}}
  \mu (M)
  \qquad \text{with }
  M := B(0,1) \cap E\big(D_{x_0}\phi(t_0+T,t_0) e^{-\alpha T}\big)
  .
\]
The semi-principle axes of the ellipsoid $E\big(D_{x_0}\phi(t_0+T,t_0) e^{-\alpha T}\big)$ have lengths
\[
  \ell_i
  :=
  \Lambda_i(t_0,x_0,T)^{-1} e^{\alpha T}
  =
  \exp(-(\lambda_i(t_0,x_0,T) - \alpha)T),
  \qquad i=1,\dots,n
  .
\]
Then $M$ contains an ellipsoid $E$ which has the same semi-principal axes as $E\big(D_{x_0}\phi(t_0+T,t_0) e^{-\alpha T}\big)$ but with lengths $\min\{1,\ell_i\}$, and is contained in a cube $C$ with side lengths $2 \min\{1,\ell_i\}$, $i=1,\dots,n$. With the volume formulas $\mu(E) = \prod_{i=1}^n \min\{1,\ell_i\} \mu(B(0,1))$, $\mu(C) = 2^n \prod_{i=1}^n \min\{1,\ell_i\}$ and $\min\{1,\ell_i\} = \exp(-(\lambda_i(t_0,x_0,T) - \alpha \big)^+ T)$, the inclusion $E \subseteq M \subseteq C$ implies $-\frac{1}{T} \log\frac{\mu(E)}{\mu(B(0,1))} \geq -\frac{1}{T} \log\frac{\mu(M)}{\mu(B(0,1))} \geq -\frac{1}{T} \log\frac{\mu(C)}{\mu(B(0,1))}$ and hence
\[
  \sum_{i=1}^n \big(\lambda_i(t_0,x_0,T) -\alpha\big)^+
  \geq
  h_{t_0}^{\alpha}(x_0)
  \geq
  \sum_{i=1}^n \big(\lambda_i(t_0,x_0,T) - \alpha\big)^+
  - \tfrac{n \log 2 + \log \alpha(\frac{n}{2}+1) - \frac{n}{2} \log \pi}{T}
  ,
\]
proving \eqref{approx}.
\end{proof}




\section{Lagrangian coherent structures based on FTME}


A commonly used tool for detection of candidates for Lagrangian coherent structures (LCS) has been the largest finite-time Lyapunov exponent (FTLE) field, whose ridges appear to mark repelling LCS (cp.\ Theorem \ref{thmFTLE} and \cite{haller1, haller1b, haller2}).

Since FTME can be also expressed in terms of FTLEs (cp.\ formulas \eqref{equ-ellipsoid}, \eqref{entropy1d}
and Proposition \ref{Pesin2D}), the ridges of an FTME field are capable of detecting candidates for LCS equally well (cp.\ Remark \ref{FTMEvsFTLE}). To illustrate this relation again in a more general context for the FTME $h^{\alpha}_{t_0}(x_0)$ with a weight $\alpha$ which depends on $t_0$ and $x_0$, let $\phi$ be a $C^2$ NDS on $X \subset J \times \R^n$ over a two-point interval $J = \{t_0, t_0 + T\}$ for some $t_0 \in \R$ and $T>0$. Define the \emph{directional stretching rate} of $\phi$ on $J$ at $(t_0,x_0)$ in direction $v \in \R^n\setminus\{0\}$ as
\begin{equation}\label{stretching-rate}
  \alpha(x_0,T,v)
  :=
  \frac{1}{T} \log \frac{\|\Phi_{x_0}(t_0 + T, t_0)v\|}{\|v\|}
\end{equation}
where $\Phi_{x_0}(t_0 + T, t_0) = D \phi(t_0 + T,t_0)x_0$. Note that with the singular vectors $\xi_i(t_0,x_0,T)$ of $\Phi_{x_0}(t_0 + T, t_0)$ we have
\[
  \alpha\big(x_0,T,\xi_i(t_0,x_0,T)\big) = \lambda_i(t_0,x_0,T)
  \quad \text{and} \quad
  \alpha(x_0,T,v) \in \big[\lambda_n(t_0,x_0,T), \lambda_1(t_0,x_0,T)\big]
  .
\]
Therefore, as a consequence, for $v \in \R^n\setminus\{0\}$
\begin{equation}
  0
  =
  h^{\lambda_1(t_0,x_0,T)}_{t_0}(x_0)
  \leq
  h^{\alpha(x_0,T,v)}_{t_0}(x_0)
  \leq
  h^{\lambda_n(t_0,x_0,T)}_{t_0}(x_0)
  .
\end{equation}
If we compute now the FTME $h_{t_0}^\alpha(x_0)$ similarly as in the proof of Theorem \ref{thm2} and choose for each $x_0 \in X(t_0)$ as the exponential weight factor $\alpha$ the directional stretching rate in direction $\xi_n(t_0,x_0,T)$ then we get
\begin{equation*}\label{smallestLE}
  h_{t_0}^{\lambda_n(t_0,x_0,T)}(x_0)
  =
  \sum_{i=1}^{n} \big( \lambda_i(t_0,x_0,T) - \lambda_n(t_0,x_0,T) \big)
  =
  \Big(
    \sum_{i=1}^n \lambda_i(t_0,x_0,T)
  \Big)
  - n \lambda_n(t_0,x_0)
  ,
\end{equation*}
because the ellipsoid $E\big(D_{x_0}\phi(t_0+T,t_0) e^{-\lambda_n(t_0,x_0,T) T}\big)$ is contained in $B(0,1)$, and formula $\eqref{equ-ellipsoid}$ yields the result. If e.g.\ $\phi$ is a two-dimensional incompressible system, i.e.\ $n=2$ and $\lambda_1(t_0,x_0,T) + \lambda_2(t_0,x_0,T) = 0$, then the weighted FTME field $x_0 \mapsto h_{t_0}^{\lambda_2(t_0,x_0,T)}(x_0) = 2 \lambda_1(t_0,x_0,T)$ is therefore proportional to the FTLE field and the search for ridge-like structures of this weighted FTME field yields LCS in the sense of Theorem \ref{thmFTLE}.

However, one major drawback of LCS based on FTLE is its inability to detect coherent structures for linear systems. This can be easily seen from the fact that for a linear differential equation $\dot x = A(t)x$ the corresponding NDS $\phi(t_0+T,t_0)$ equals its linearization $\Phi(t_0+T,t_0) := D \phi(t_0+T,t_0)x_0$ which is independent of $x_0$. Consequently the FTLEs $\lambda_i(t_0,x_0,T) \equiv \lambda_i(t_0,T)$ are also independent of $x_0$. There might exist weak LCS in the sense of Theorem \ref{thmFTLE} but no LCS, since the FTLE field is constant and therefore condition (ii)2 of Theorem \ref{thmFTLE} is not satisfied (for a discussion of limitations of LCS based on FTLE see \cite{ross}). In \cite[Section 9]{haller1} Haller developed a notion of constrained LCS for autonomous systems. It would be interesting to investigate whether constrained LCS for autonomous systems are capable of detecting classical stable and unstable manifolds of equilibria (cf.\ Theorem \ref{autonomentropy} below).

In this section we introduce and discuss LCS based on FTME for autonomous differential equations
\begin{equation}\label{autonom1}
  \dot{x} = f(x),
  \qquad
  t \in [0,T], x \in U \subseteq \R^n,
\end{equation}
with a $C^2$ function $f : U \rightarrow \R^n$ and $T > 0$. Assume that for all $x_0$ in some $X(0) \subseteq U$ the solution $x(\cdot,x_0)$ which starts at time $0$ in $x_0$ exists on the whole interval $[0,T]$ and define $X(T) := x(T,X(0))$. Then $\phi(t,s,\cdot) : X(s) \rightarrow X(t)$, $\phi(t,s,x_0) := x(t-s,x_0)$, is an NDS on $X := \{(t,x) \in J \times U : x \in X(t)\}$ over the two-point set $J := \{0,T\}$ (cp.\ Example \ref{example-NDS}(b)). Since $\phi(t,s,\cdot)$ depends only on the difference $t-s$ we write instead for simplicity $\phi(t-s,\cdot)$, and similarly for its linearization $\Phi_{x_0}(t-s)$ (cp.\ Example \ref{example-NDS}(c)). We write $h^\alpha(x_0)$ for $h_{t_0}^\alpha(x_0)$ if $t_0 = 0$.

To compute LCS based on FTME, we study $h^{\alpha(x_0,T,f(x_0))}(x_0)$ at each $x_0 \in X(0)$ and use as the exponential weight factor the directional stretching rate of $\phi$ in the direction of the vector field $f$
\begin{equation}\label{strechingrate2}
  \alpha(x_0,T,f(x_0))
  =
  \frac{1}{T} \log \frac{\|\Phi_{x_0}(T)f(x_0)\|}{\|f(x_0)\|}
  ,
\end{equation}
where we simplified the notation by omitting the initial time $t_0=0$ in \eqref{stretching-rate}. Note that with this approach we emphasize the direction of the vector field when it comes to measuring attraction and repulsion rates. In comparison, Haller \cite{haller1} measures growth rates in directions normal to potential LCS manifolds. In fact his concept of repulsion ratio is the quotiont of his repulsion rate and the maximum over all directions $v \neq 0$ of our directional stretching rate. Using the fact that $t \mapsto \Phi_{x_0}(t) f(x_0)$, as well as $t \mapsto f(\phi(t,x_0))$, are solutions to the same initial value problem $\dot x = Df(\phi(t,x_0))x$, $x(0) = f(x_0)$, it follows that $\Phi_{x_0}(t) f(x_0) = f(\varphi(t,x_0))$ for $t \in [0,T]$, $x_0 \in X(0)$, and hence the expression \eqref{strechingrate2} for the directional stretching rate in the direction of the vector field equals   $\alpha(x_0,T,f(x_0)) =   \frac{1}{T} \log \frac{\|f(\phi(T,x_0))\|}{\|f(x_0)\|}$.

We show now for two classes of examples that the weighted FTME field
\begin{equation}\label{FTMEfield}
  x \mapsto H(x) := h^{\alpha(x,T,f(x))}(x)
  \qquad \text{with} \quad
  \alpha(x,T,f(x)) = \frac{1}{T} \log \frac{\|f(\phi(T,x))\|}{\|f(x)\|}
\end{equation}
exhibits ridge and trough-like coherent structures which approach classical invariant manifolds as $T \to \infty$.
%
%
Figure \ref{fig:linear1} shows the weighted FTME field \eqref{FTMEfield} for the linear differential equation $\dot x_1 =  x_1-x_2$, $\dot x_2 = - x_2$, with stable manifold $W^s = \{(x_1,x_2) \in \R^2 : 2x_1 = x_2\}$ and unstable manifold $W^u = \{(x_1,x_2) \in \R^2 : x_2 = 0\}$. Figure \ref{fig:parabol1} shows the weighted FTME field \eqref{FTMEfield} for the nonlinear differential equation $\dot x_1 = - x_1$, $\dot x_2 = x_1^2 + x_2$. Its unstable manifold $W^u = \{ (x_1,x_2) \in \R^2 : x_1=0\}$ is the $x_2$-axis and its stable manifold $W^u = \{(x_1,x_2) \in \R^2 : x_2 = \frac{1}{3} x_1^2\}$ is a parabola.

Instead of developing a complete theory for ridge-like structures of weighted FTME fields in this section, we take advantage of the fact that for the specific classes of examples which we discuss in this section, the points $x$ in the ridge and trough-like structures of the weighted FTME field \eqref{FTMEfield} satisfy the condition $\nabla H(x) = 0$. Using Proposition \ref{Pesin2D}(ii) with the abbreviations $\kappa_1(x) := e^{(\lambda_1(x,T)-\alpha(x,T,f(x))) T}$ and
$\kappa_2(x) := e^{(\lambda_2(x,T)-\alpha(x,T,f(x))) T}$, and using the fact that $\alpha(x,T,f(x)) \in [\lambda_2(x,T), \lambda_1(x,T)]$ and hence $\kappa_1(x) > 1 > \kappa_2(x)$, we get $H(x) = h(\kappa_1(x), \kappa_2(x))$ with
$h(\kappa_1, \kappa_2) = - \frac{1}{T} \log \frac{2}{\pi} \left( \arccos \sqrt{\frac{\kappa_1^2-1}{\kappa_1^2 - \kappa_2^2}} + \frac{1}{\kappa_1 \kappa_2} \arccos \kappa_1 \sqrt{\frac{1-\kappa_2^2}{\kappa_1^2 - \kappa_2^2}}\right)$ and hence
\begin{eqnarray}\label{GradientEntropy}
  \nabla H(x)
  & = &
  \tfrac{\partial h}{\partial \kappa_1}\big(\kappa_1(x), \kappa_2(x)\big)
  \kappa_1(x) T \big( \nabla \lambda_1(x,T) - \nabla \alpha(x,T,f(x)) \big)
  \nonumber
\\
  & & {} +
  \tfrac{\partial h}{\partial \kappa_2}\big(\kappa_1(x), \kappa_2(x)\big)
  \kappa_2(x) T \big( \nabla \lambda_2(x,T) - \nabla \alpha(x,T,f(x)) \big)
  .
\end{eqnarray}
Note that in higher dimensions the weighted FTME field \eqref{FTMEfield} could fail to be $C^1$ for $x$ in a lower-dimensional subset (cp.\ Kato \cite{kato}).

We will present two examples below which both satisfy the assumption that $\nabla \big(\lambda_1(x,T) + \lambda_2(x,T)\big) = 0$. In this case, finding the zeros of $\nabla H(x)$ in \eqref{GradientEntropy} is equivalent to solving
\begin{eqnarray}\label{localextrema}
  & &
  \Big( \frac{\partial h}{\partial \kappa_1}\big(\kappa_1(x), \kappa_2(x)\big) \kappa_1(x)
  -
  \frac{\partial h}{\partial \kappa_2}\big(\kappa_1(x), \kappa_2(x)\big) \kappa_2(x) \Big)
  \nabla \lambda_1(x,T)
  \nonumber
\\
  & = &
  \Big(
  \frac{\partial h}{\partial \kappa_1}\big(\kappa_1(x), \kappa_2(x)\big) \kappa_1(x)
  +
  \frac{\partial h}{\partial \kappa_2}\big(\kappa_1(x), \kappa_2(x)\big) \kappa_2(x) \Big)
  \nabla \alpha(x,T,f(x))
  .
\end{eqnarray}

\begin{example}\label{exlinear}
Consider a two dimensional linear autonomous system
\begin{equation}\label{ex2Dlinear}
  \dot{x} = A x
\end{equation}
with a matrix $A \in \R^{2\times 2}$ which has two eigenvalues $\lambda_1 > 0 > \lambda_2$ and corresponding  eigenvectors $e_1, e_2$. The solution of the initial value problem \eqref{ex2Dlinear}, $x(0) = x_0$, is $\varphi(t,x_0) = e^{A t} x_0$, and the system has unstable and stable manifolds $W^u$ and $W^s$ which are two lines with directional vectors $e_1$ and $e_2$. The linearized solution $\Phi(t) = \Phi_{x_0}(t) := D_{x_0} \phi(t,x_0) = e^{A t}$ is independent of $x_0$. We consider \eqref{ex2Dlinear} for $t \in [0,T]$ for arbitrary $T>0$. It follows that the FTLE $\lambda_i(T) = \lambda_i(x_0,T)$ are also independent of $x_0$, i.e.\ $\nabla \lambda_i(x_0,T) = (0,0)$ for $i=1,2$. In particular the FTLE field $x_0 \mapsto \lambda_1(x_0,T)$ is constant and not capable of detecting any coherent structures such as stable or unstable manifolds.

\begin{figure}[ht]
\centering
\begin{minipage}[b]{0.45\linewidth}
\centerline{\includegraphics[width=10cm]{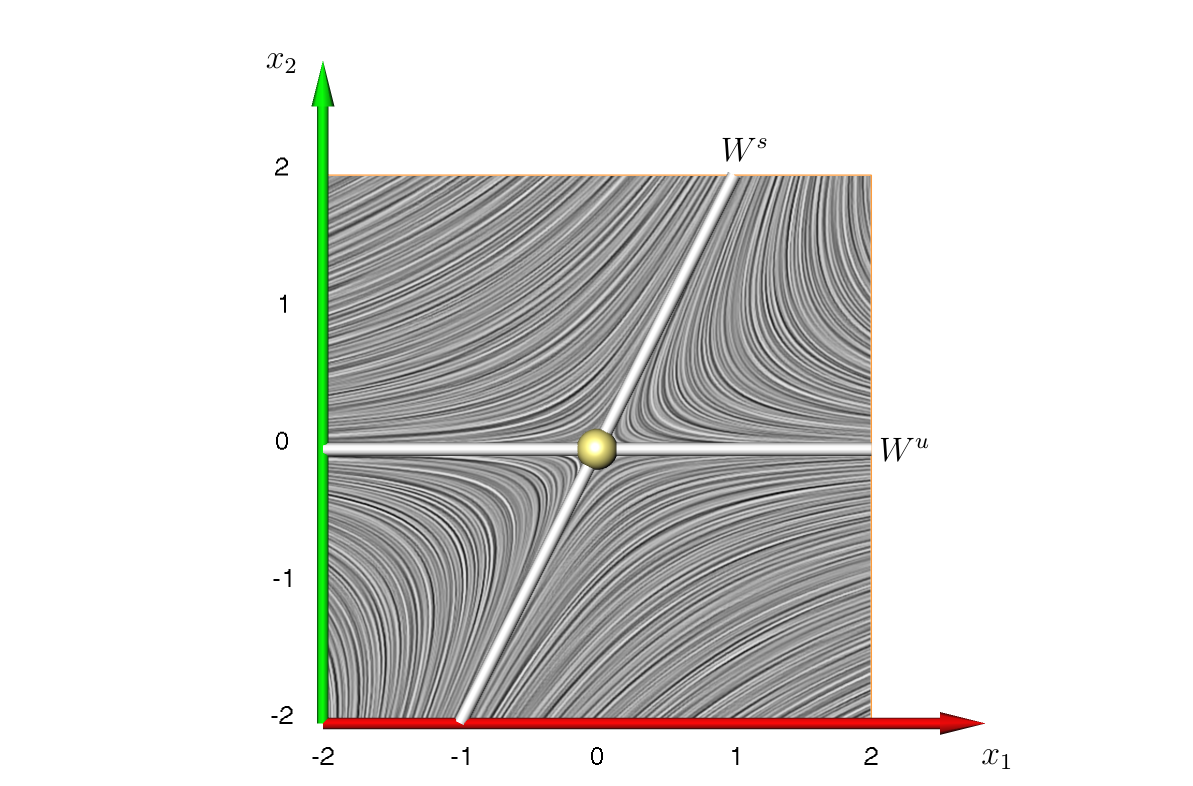}}
\caption{Vector field and invariant manifolds $W^s$ and $W^u$ for $\dot x_1 = x_1 - x_2$, $\dot x_2 = - x_2$.}
\label{fig:linear1}
\end{minipage}
\hfill
\begin{minipage}[b]{0.45\linewidth}
\centerline{\includegraphics[width=10cm]{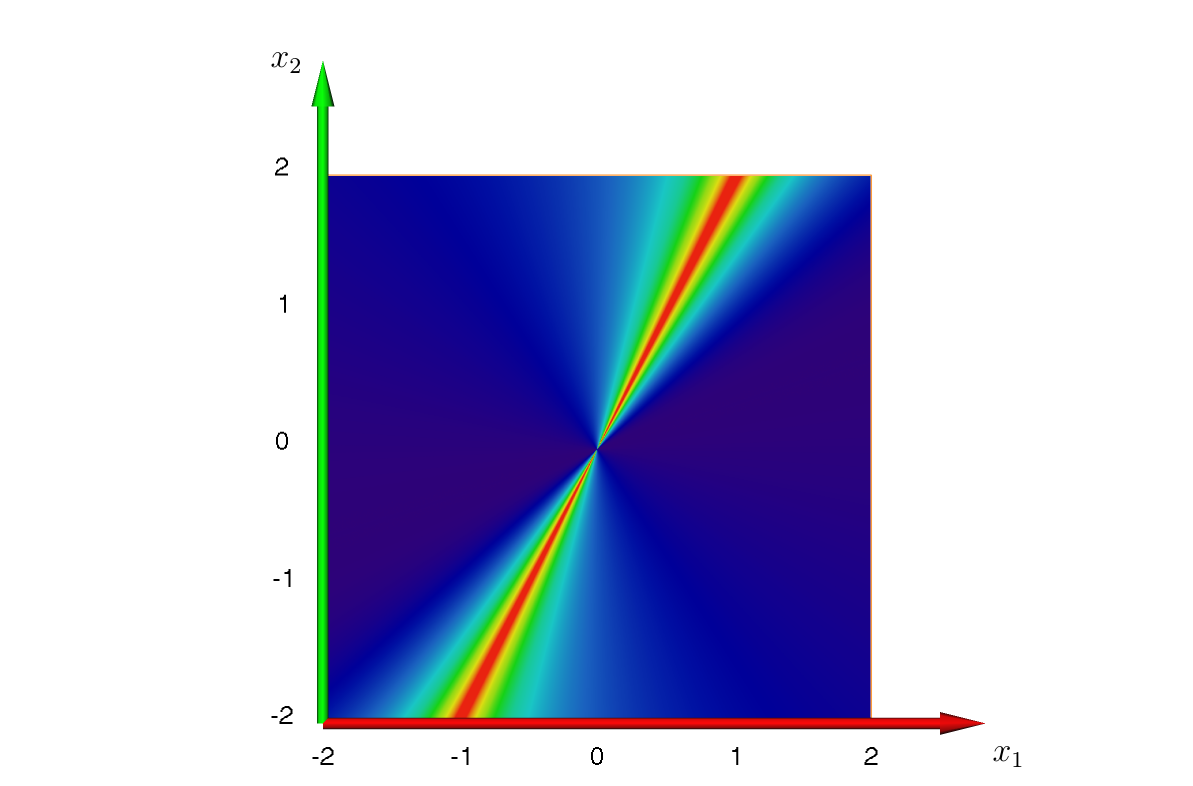}}
\caption{Weighted FTME field \eqref{FTMEfield} for $\dot x_1 = x_1 - x_2$, $\dot x_2 = - x_2$ (blue $\simeq 0$, red $\simeq 1$).}
\label{fig:linear2}
\end{minipage}
\end{figure}

The solutions $x_0$ of \eqref{localextrema} are given by the zeros of $\nabla \alpha(x_0)$, where $\alpha(x_0) = \alpha(x_0,T,Ax_0) = \frac{1}{T} \log \frac{\|\Phi(T)Ax_0\|}{\|Ax_0\|}$. It follows that
\begin{eqnarray*}
  0
  &=&
  \nabla \alpha(x_0)
  =
  \frac{1}{T} \frac{\|Ax_0\|}{\|\Phi(T)Ax_0\|} \nabla\Big(\frac{\|\Phi(T)Ax_0\|}{\|Ax_0\|}\Big)
\\
  &=&
  \frac{1}{T} \frac{1}{\|\Phi(T)Ax_0\|^2} A^{\rT} \Big( \Phi(T)^{\rT}\Phi(T) Ax_0 -
  \frac{\|\Phi(T)Ax_0\|^2}{\|Ax_0\|^2} Ax_0\Big)
\end{eqnarray*}
and hence $Ax_0$ is an eigenvector of $\Phi(T)^{\rT} \Phi(T)$. Let $U = (e_1 | e_2)$ denote the matrix whose column vectors are $e_1$ and $e_2$, thus $A = U \diag(\lambda_1,\lambda_2)  U^{-1}$. We then have
\begin{eqnarray*}
  \lambda_1(T)
  &=&
  \frac{1}{T} \log \|\Phi(T)\| = \frac{1}{T} \log \|U \diag(e^{\lambda_1T}, e^{\lambda_2 T}) U^{-1}\|
  ,
\\
  \lambda_2(T)
  &=&
  \frac{1}{T} \log \|\Phi(T)^{-1}\| = \frac{1}{T} \log \|U \diag(e^{-\lambda_1T}, e^{-\lambda_2 T}) U^{-1}\|
  .
\end{eqnarray*}
A direct computation shows that for $\eps>0$ and $T > \frac{\log \|U\| + \log \|U^{-1}\|}{\eps}$ the estimate
\begin{equation*}
  |\lambda_i(T) - \lambda_i|
  \leq
  \frac{1}{T} \big(\log \|U\| + \log \|U^{-1}\|\big) \leq \eps
  \qquad \text{for } i = 1,2
\end{equation*}
holds.

For $i=1,2$ let $v_i(T) = v_i(x_0,T)$ with $\|v_i(T)\| = 1$ denote the finite-time Lyapunov vectors corresponding to $\lambda_i(T)$, i.e.\ $\Phi(T)^{\rT} \Phi(T) v_i(T) = \lambda_i(T) v_i(T)$. Since $e_2 = \langle e_2, v_1(T) \rangle v_1(T) + \langle e_2, v_2(T) \rangle v_2(T)$, we get
\begin{eqnarray*}
  e^{2 \lambda_2 T}
  &=&
  \|\Phi(T) e_2\|^2 = \langle e_2, \Phi(T)^{\rT} \Phi(T) e_2\rangle
\\
  &=&
  \langle e_2, v_1(T)\rangle^2 e^{2 \lambda_1(T) T} +
  \langle e_2, v_2(T)\rangle^2 e^{2 \lambda_2(T) T}.
\end{eqnarray*}
Hence $|\langle e_2,v_1(T)\rangle| \leq e^{(\lambda_2 - \lambda_1(T)) T} \leq e^{(\lambda_2 - \lambda_1 + \eps) T}$ or $|\sin \angle (e_2,v_2(T))| \leq e^{(\lambda_2 - \lambda_1 + \eps) T}$.

Assume that $Ax_0 \parallel v_2(T)$, the case $Ax_0 \parallel v_1(T)$ is analog. Then $|\sin \angle (e_2, \frac{Ax_0}{\|Ax_0\|})| \leq e^{(\lambda_2 - \lambda_1 + \eps) T}$. Since $A e_2 = \lambda_2 e_2$, it follows that
$|\sin \angle (e_2,\frac{x_0}{\|x_0\|})| \leq e^{(\lambda_2 - \lambda_1 + \epsilon) T}$. In other words, the ridge and trough of the weighted FTME field $x_0 \mapsto H(x_0)$ of system \eqref{ex2Dlinear} are the two lines which are defined by the zeros of $\nabla H(x_0)$ and converge to the stable and unstable invariant manifolds for $T \to \infty$, see also the vector field in Figure \ref{fig:linear1} and the weighted FTME field of \eqref{ex2Dlinear} in Figure \ref{fig:linear2} for $A = {1 \; -1 \choose 0 \; -1}$.
\end{example}


\begin{example}\label{parabola}
For $\gamma, \beta \in \R$, $\gamma > 0$, consider the following family of autonomous differential equations
\begin{equation}\label{ex1}
  \begin{array}{rcl}
    \dot{x}_1 &=& -x_1,
  \\
    \dot{x}_2 &=& \beta x_1^2 + \gamma x_2.
  \end{array}
\end{equation}
Its solution for $t \in \R$ and $x_0 = (x_{01}, x_{02}) \in \R^2$ is given by
\[
  \phi(t,x_0) =
  \left\{
  \begin{array}{l}
    e^{-t}x_{01}
  \\
    \frac{\beta}{2+\gamma}x_{01}^2 (e^{\gamma t} - e^{-2t}) + e^{\gamma t} x_{02}
  \end{array}
  \right.
\]
and its linearization $\Phi_{x_0}(t) = D_{x_0} \phi(t,x_0)$ is
\begin{eqnarray*}
  \Phi_{x_0}(t)
  =
  \left(
    \begin{array}{ccc}
      e^{-t} & 0
    \\
      \frac{2\beta}{2+\gamma}x_{01}(e^{\gamma t} - e^{-2t}) & e^{\gamma t}
    \end{array}
  \right)
  =:
  \left(
    \begin{array}{ccc}
      a & 0
    \\
      b & c
    \end{array}
  \right)
  .
\end{eqnarray*}
System \eqref{ex1} has an equilibrium at the origin with invariant stable and unstable manifolds
\begin{eqnarray*}
  W^s
  &=&
  \big\{(x_{01},x_{02}) \in \R^2 : x_{02} + \tfrac{\beta}{2+\gamma} x_{01}^2 = 0\big\},
\\
  W^u
  &=&
  \big\{(x_{01},x_{02}) \in \R^2 : x_{01} = 0 \big\}.
\end{eqnarray*}
We consider \eqref{ex1} for $t \in [0,T]$ for an arbitrary $T>0$. Its FTLEs $\lambda_i(x_0,T)$ for $i=1,2$ are
\[
  \lambda_{i}(x_0,T)
  =
  \frac{1}{2 T} \log \Lambda_{i}(x_0,T)
  \quad \text{with }
  \Lambda_{i}(x_0,T)
  =
  \frac{a^2 + b^2 + c^2 \pm [(a^2 + b^2 + c^2)^2 - 4]^{\frac{1}{2}}}{2}.
\]
Thus $\lambda_1(x_0,T)$ depends only on $x_{01}$, i.e.\ $\frac{\partial \lambda_1(x_0,T)}{\partial x_{02}} = 0$. By  \eqref{localextrema}, the zeros of $\nabla H(x_0)$ are the solutions to
\begin{eqnarray}\label{localextrema15}
  & &
  \Big( \frac{\partial h}{\partial \kappa_1}\big(\kappa_1(x), \kappa_2(x)\big) \kappa_1(x)
  -
  \frac{\partial h}{\partial \kappa_2}\big(\kappa_1(x), \kappa_2(x)\big) \kappa_2(x) \Big)
  \frac{\big(\frac{\partial \Lambda_1(x_0,T)}{\partial x_{01}},0\big)}{2T \Lambda_1(x_0,T)}
  \nonumber
\\
  & = &
  \Big(
  \frac{\partial h}{\partial \alpha_1}\big(\alpha_1(x), \alpha_2(x)\big) \alpha_1(x)
  +
  \frac{\partial h}{\partial \alpha_2}\big(\alpha_1(x), \alpha_2(x)\big) \alpha_2(x) \Big)
  \frac{\big(\frac{\partial g(x_0)}{\partial x_{01}}, \frac{\partial g(x_0)}{\partial x_{02}}\big)}{2T g(x_0)}
  ,
\end{eqnarray}
where
\[
  g(x_0)
  =
  \frac{\|\Phi_{x_0}(T) f(x_0)\|^2}{\|f(x_0)\|^2}
  =
  \frac{e^{-2T}x_{01}^2 +
  \big[\gamma e^{\gamma T}x_{02} +
  \frac{\gamma \beta}{2 + \gamma} (e^{\gamma T} - e^{-2T})x_{01}^2 + \beta e^{-2T} x_{01}^2\big]^2}{x_{01}^2 +
  (\gamma x_{02} + \beta x_{01}^2)^2}
  .
\]
Equation \eqref{localextrema15} implies $\frac{\partial g}{\partial x_{02}} = 0$, which yields
\begin{eqnarray}\label{parabol3}
  & &
  2 \gamma e^{\gamma T} \big[\gamma e^{\gamma T}x_{02} +
  \tfrac{\gamma \beta}{2 + \gamma} (e^{\gamma T} -
  e^{-2T})x_{01}^2 +
  \beta e^{-2T} x_{01}^2\big] \big[x_{01}^2 + (\gamma x_{02} + \beta x_{01}^2)^2\big]
\\
  \nonumber
 & = &
 2 \gamma (\gamma x_{02} + \beta x_{01}^2) \big\{e^{-2T}x_{01}^2 +
 \big[\gamma e^{\gamma T}x_{02} + \tfrac{\gamma \beta}{2 +
 \gamma} (e^{\gamma T} - e^{-2T})x_{01}^2 + \beta e^{-2T} x_{01}^2\big]^2\big\}
 .
\end{eqnarray}
Dividing both sides of \eqref{parabol3} by $e^{2\gamma T}$, we get
\begin{equation}\label{parabol4}
  x_{01}^2 \big(x_{02} +
  \tfrac{\beta}{2+\gamma} x_{01}^2\big)
  \big(x_{02} + \tfrac{\beta}{\gamma}x_{01}^2 +
  \tfrac{2+\gamma}{2 \gamma \beta}\big)
  =
  e^{-(1+2\gamma)T} h(x_{01},x_{02}),
\end{equation}
where $h(x_{10},x_{20})$ is a polynomial in $x_{10}$ and $x_{20}$. Since the right hand side of \eqref{parabol4} tends to $0$ as $T \to \infty$, the solutions $x_0 = (x_{01}, x_{02})$ of equation \eqref{parabol4} satisfies either
$x_{01} \approx 0$ or $x_{02} + \frac{\beta}{2+\gamma} x_{01}^2 \approx 0$ or $x_{02} + \frac{\beta}{\gamma}x_{01}^2 + \frac{2+\gamma}{2 \gamma \beta} \approx 0$ for $T$ large. Finally, by solving the first component of \eqref{localextrema15}, we conclude that the zeros $x_0$ of $\nabla H(x_0)$ satisfy either $x_{01} \to 0$ or $x_{02} + \frac{\beta}{2+\gamma} x_{01}^2 \to 0$ for $T \to \infty$, proving that the rigde and trough-like structures of the weighted FTME field $x_0 \mapsto H(x_0)$ are finite-time approximations of the stable and unstable manifold, see also the vector field and weighted FTME field of \eqref{parabol4} for $\gamma = \beta = 1$ in Figures \ref{fig:parabol1} and \ref{fig:FTME}.

\begin{figure}[ht]
\centering
\begin{minipage}[b]{0.45\linewidth}
\centerline{\includegraphics[width=10cm]{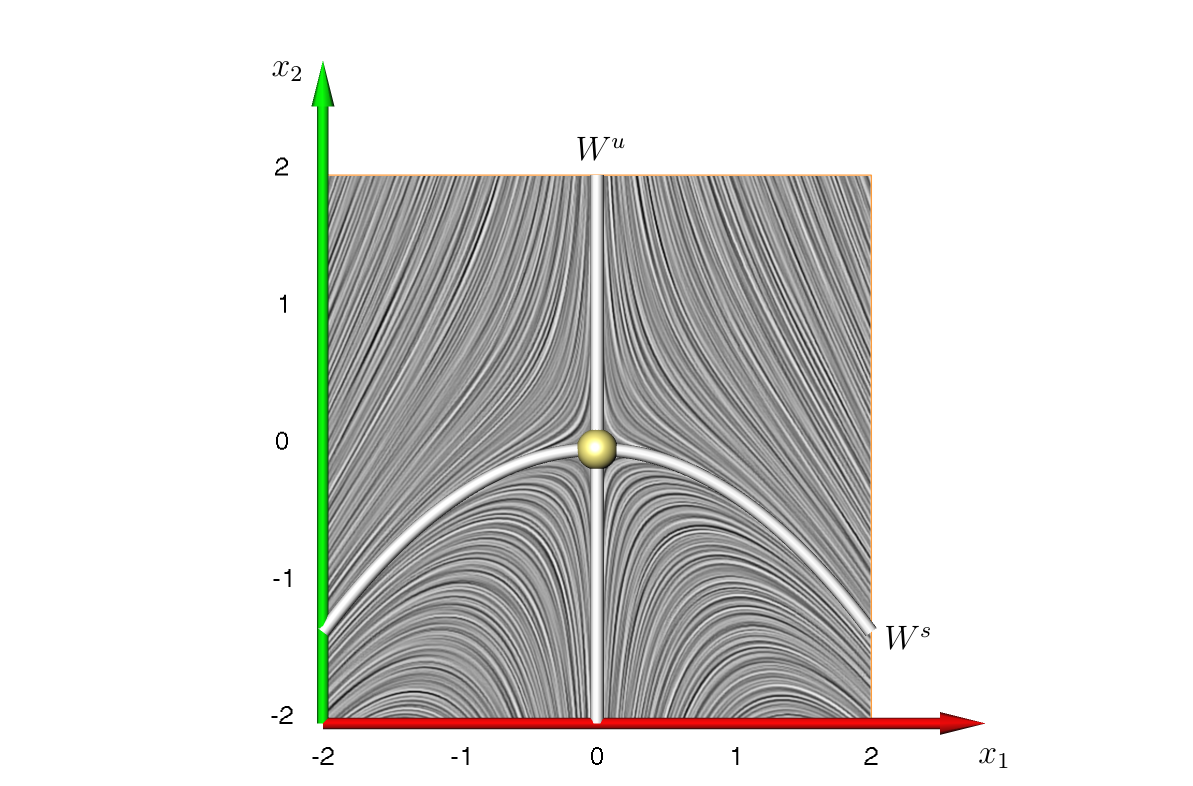}}
\caption{Vector field and invariant manifolds $W^s$ and $W^u$ for $\dot x_1 = - x_1$, $\dot x_2 = x_1^2 + x_2$.}
\label{fig:parabol1}
\end{minipage}
\hfill
\begin{minipage}[b]{0.45\linewidth}
\centerline{\includegraphics[width=10cm]{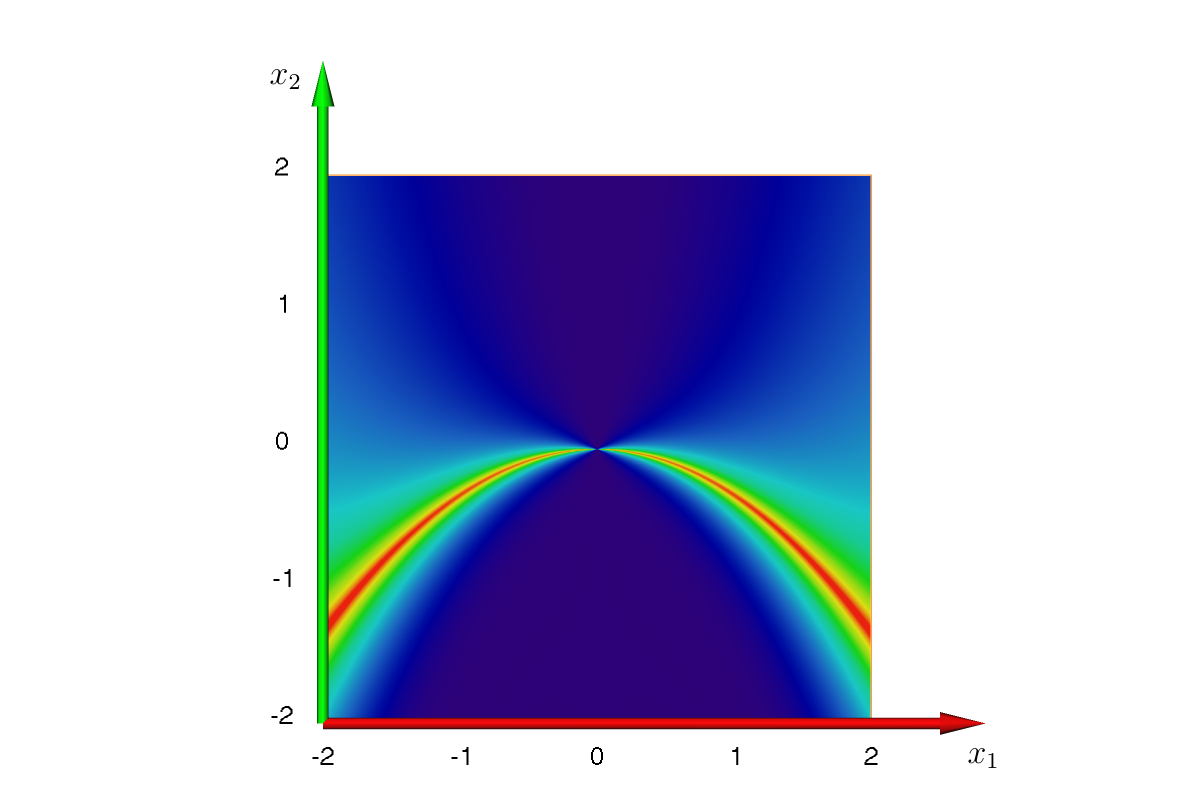}}
\caption{Weighted FTME field \eqref{FTMEfield} for $\dot x_1 = - x_1$, $\dot x_2 = x_1^2 + x_2$ with $T = 2$ (blue $\simeq 0$, red $\simeq 2.25$).}
\label{fig:FTME}
\end{minipage}
\end{figure}



As can be seen in Figures \ref{fig:FTLE1} and \ref{fig:FTLE2}, the forward and backward FTLE fields are not capable of detecting the stable manifold of \eqref{parabol4} for $\gamma = \beta = 1$. In fact, a smooth compact curve $\cM(t) \subset \R^2$ at time $t_0$ needs to satisfy conditions 2 and 3 of Theorem \ref{thmFTLE}(i) in order to qualify as a candidate for a weak LCS. This is equivalent to
$\nabla \lambda_1(x_0,T)^{\rT} \parallel T_{x_0}M(t_0)$. Since $\frac{\partial \lambda_1(x_0,T)}{\partial x_{02}} = 0$, this is equivalent to $T_{x_0}\cM(t_0) \parallel (1,0)^{\rT}$. Hence the possible repelling LCS candidates can only be lines which are parallel to the $x_{01}$ axis, in constrast to the stable manifold $W^s = \big\{(x_{01},x_{02}) \in \R^2 : x_{02} + \tfrac{\beta}{2+\gamma} x_{01}^2 = 0\big\}$ of \eqref{parabol4} which is a parabola.

\begin{figure}[ht!]
\centering
\begin{minipage}[b]{0.45\linewidth}
\centerline{\includegraphics[width=10cm]{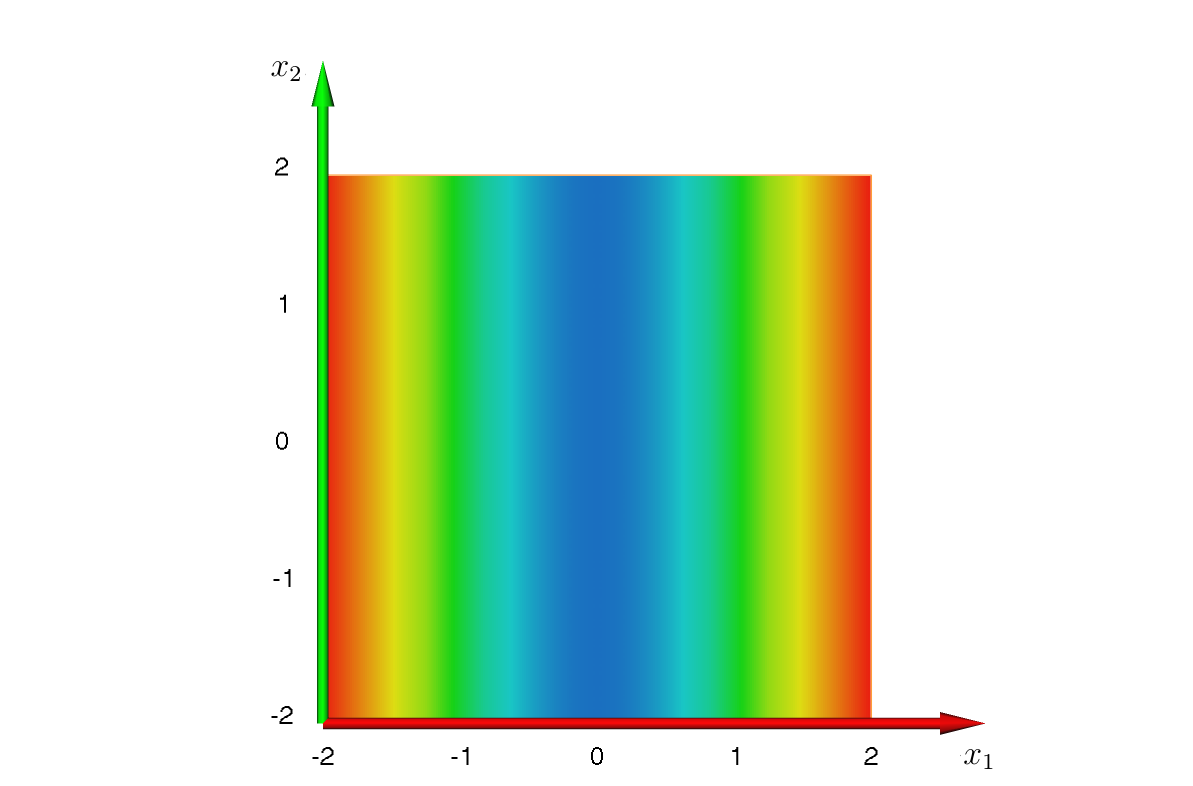}}
\caption{Forward FTLE field $x_0 \mapsto \lambda_1(x_0,T)$ for $\dot x_1 = - x_1$, $\dot x_2 = x_1^2 + x_2$ with $T=2$ (blue $\simeq 0.9$, red $\simeq 1.25$).}
\label{fig:FTLE1}
\end{minipage}
\hfill
\begin{minipage}[b]{0.45\linewidth}
\centerline{\includegraphics[width=10cm]{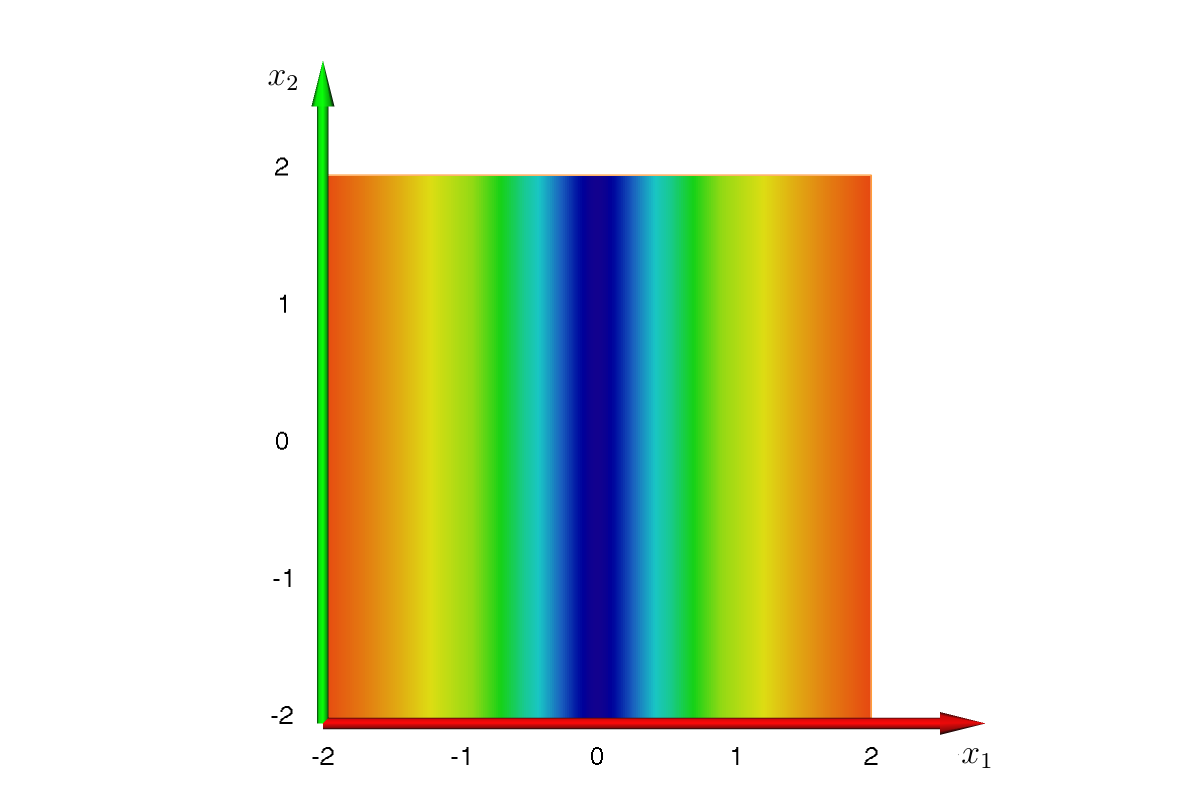}}
\caption{Backward FTLE field for $\dot x_1 = - x_1$, $\dot x_2 = x_1^2 + x_2$ on $[-2,0]$ (blue $\simeq 0.9$, red $\simeq 2.2$).}
\label{fig:FTLE2}
\end{minipage}
\end{figure}

\end{example}

The two examples \ref{exlinear} and \ref{parabola} illustrate that ridge and trough-like structures of the weighted FTME field approximate classical stable and unstable manifolds. In the remainder of this section we show that also for arbitrary two-dimensional systems \eqref{autonom1} the weighted FTME field is capable of detecting invariant manifolds in the vicinity of equilibria.
The following preparatory lemma provides an estimate for the stretching rate \eqref{strechingrate2} near an isolated equilibrium $x^*$.


\begin{lemma}[Directional stretching rate in direction of vector field close to equilibrium]\label{lemma1}
Assume that \eqref{autonom1} has an isolated equilibrium $x^*$. Then its directional stretching rate $\alpha(x_0) := \alpha(x_0,T,f(x_0))$ for $x_0 \neq x^*$ can be approximated in the following sense
\begin{equation}\label{gammaAprox}
  \lim_{x_0 \to x^*}
  \Big|\alpha(x_0) - \frac{1}{T} \log \frac{\|\Phi_{x^*}(T)(x_0-x^*)\|}{\|x_0-x^*\|}\Big|
  \leq
  \frac{1}{T} \big(\log \|D_xf(x^*)\| + \log \|D_xf(x^*)^{-1}\|\big)
  .
\end{equation}
\end{lemma}

\begin{proof}
To estimate the stretching rate, observe that for $x_0 \ne x^*$ with $f(x_0) \neq 0$
\begin{eqnarray*}
  \alpha(x_0)
  &=&
  \frac{1}{T} \log \frac{\|\Phi_{x_0}(T)f(x_0)\|}{\|f(x_0)\|}
\\
  &=&
  \frac{1}{T} \log \Big \| \frac{\Phi_{x^*}(T) f(x_0) }{\|f(x_0)\|}
  +\frac{ [\Phi_{x_0}(T) - \Phi_{x^*}(T)]f(x_0)}{\|f(x_0)\|} \Big\|
\\
  &=&
  \frac{1}{T} \log \Big \| \frac{\Phi_{x^*}(T) f(x_0) }{\|f(x_0)\|}
  + g(T,x_0-x^*)\frac{f(x_0)}{\|f(x_0)\|} \Big\|
\\
  &=&
  \frac{1}{T} \log \Big \| \frac{\Phi_{x^*}(T)[f(x^*) + D_xf(x^*)(x_0-x^*) + h(T,x_0-x^*)(x_0-x^*)]}
  {\|f(x^*) + D_xf(x^*)(x_0-x^*) + h(T,x_0-x^*)(x_0-x^*)\|}
\\
  & &
  {} + g(T,x_0-x^*)\frac{f(x_0)}{\|f(x_0)\|} \Big\|
\\
  &=&
  \frac{1}{T} \log \Big \| \frac{\Phi_{x^*}(T)[D_xf(x^*)(x_0-x^*) + h(T,x_0-x^*)(x_0-x^*)]}
  {\|D_xf(x^*)(x_0-x^*) + h(T,x_0-x^*)(x_0-x^*)\|}
\\
  & &
  {} + g(T,x_0-x^*)\frac{f(x_0)}{\|f(x_0)\|} \Big\|
\\
  &=&
  \frac{1}{T} \log \Big \| \frac{\Phi_{x^*}(T)[\frac{D_xf(x^*)(x_0-x^*)}{\|D_xf(x^*)(x_0-x^*)\|}
  + \frac{h(T,x_0-x^*)(x_0-x^*)}{\|D_xf(x^*)(x_0-x^*)\|}]}
  {\|\frac{D_xf(x^*)(x_0-x^*)}{\|D_xf(x^*)(x_0-x^*)\|}
  + \frac{h(T,x_0-x^*)(x_0-x^*)}{\|D_xf(x^*)(x_0-x^*)\|}\|}
  + g(T,x_0-x^*)\frac{f(x_0)}{\|f(x_0)\|} \Big\|
\end{eqnarray*}
where $g(T,x_0-x^*), h(T,x_0-x^*) \in \R^n$ satisfy $\lim \limits_{x_0 \to x^*}g(T,x_0-x^*) = \lim \limits_{x_0 \to x^*}h(T,x_0-x^*) = 0$. Since $\lim \limits_{x_0 \to x^*} \|g(T,x_0-x^*)\frac{f(x_0)}{\|f(x_0)\|}\| = 0$ and
\[
\lim \limits_{x_0 \to x^*} \frac{h(T,x_0-x^*)(x_0-x^*)}{\|D_xf(x^*)(x_0-x^*)\|}  = \lim \limits_{x_0 \to x^*} \frac{h(T,x_0-x^*)\frac{x_0-x^*}{\|x_0-x^*\|}}{\|D_xf(x^*)\frac{(x_0-x^*)}{\|x_0-x^*\|}\|} =0,
\]
we have
\[
\lim_{x_0\to x^*} \alpha(x_0) - \frac{1}{T} \log \frac{\|\Phi_{x^*}(T) D_x f(x^*)(x_0-x^*)\|}{\|D_x f(x^*)(x_0-x^*)\|} = 0.
\]

Using the fact that $\Phi_{x^*}(T)D_xf(x^*) = e^{D_xf(x^*)T}D_xf(x^*) = D_xf(x^*)e^{D_xf(x^*)T} = D_xf(x^*)\Phi_{x^*}(T) $ and the estimates
\begin{eqnarray*}
\frac{\|\Phi_{x^*}(T)(x_0-x^*)\|}{\|D_xf(x^*)^{-1}\|} &\leq& \|D_xf(x^*) \Phi_{x^*}(T)(x_0-x^*) \| \leq \|D_xf(x^*)\| \| \Phi_{x^*}(T)(x_0-x^*) \|, \\
\frac{\|x_0-x^*\|}{\|D_xf(x^*)^{-1}\|} &\leq& \|D_xf(x^*) (x_0-x^*)\| \leq \|D_xf(x^*)\| \|(x_0-x^*)\|,
\end{eqnarray*}
it follows that
\begin{eqnarray*}
\frac{\|\Phi_{x^*}(T)(x_0-x^*)\|}{\|D_xf(x^*)\| \|D_xf(x^*)^{-1}\| \|x_0-x^*\|} &\leq& \frac{\|D_xf(x^*) \Phi_{x^*}(T)(x_0-x^*) \|}{\|D_xf(x^*) (x_0-x^*)\| }\\
&\leq& \|\Phi_{x^*}(T)(x_0-x^*)\| \|D_xf(x^*)\| \|D_xf(x^*)^{-1}\|,
\end{eqnarray*}
or equivalently
\begin{eqnarray}\label{eq08}
&& \Big |\frac{1}{T} \log \frac{\|\Phi_{x^*}(T) D_x f(x^*)(x_0-x^*)\|}{\|D_x f(x^*)(x_0-x^*)\|} - \frac{1}{T} \log \frac{\|\Phi_{x^*}(T)(x_0-x^*)\|}{\|x_0-x^*\|} \Big | \\\nonumber
&& \leq \frac{1}{T} \log \Big(\|D_xf(x^*)\| \|D_xf(x^*)^{-1}\|\Big),
\end{eqnarray}
proving \eqref{gammaAprox}.
\end{proof}


The following theorem states that for planar systems \eqref{autonom1} minima and maxima of the weighted FTME field \eqref{FTMEfield} in the vicinity of an equilibrium indicate Lagrangian coherent structures which locally approximate the classical unstable and stable manifolds. The theorem also holds in higher dimensions for one-dimensional strongly unstable and strongly stable manifolds.


\begin{theorem}\label{autonomentropy}
Consider a two-dimensional system \eqref{autonom1} on an open set $U \subset \R^2 $ and assume that it has an isolated equilibrium $x^* \in U$ which is hyperbolic. Let $\lambda_1 > 0 > \lambda_2$ denote the eigenvalues of $D_xf(x^*)$ with corresponding normalized eigenvectors $e_1, e_2$. Then for all $\eps \in (0, \min\{\lambda_1, -\lambda_2\})$ there exists $T(\eps) > 0$, satisfying that for all $T \geq T(\eps)$ there exists a $\delta(T) > 0$ such that the weighted FTME field $x \mapsto H(x) := h^{\alpha(x,T,f(x))}(x)$ in formula \eqref{FTMEfield} satisfies the following properties:
\begin{itemize}
  \item[(i)] \emph{Bound for weighted FTME field:}
    $H(x) \in [0, \lambda_1 -\lambda_2 + \epsilon)$ for $x \in B(x^*,\delta(T))$

  \item[(ii)] \emph{Unstable cone contains unstable manifold and has minimal FTME values:}
    The so-called \emph{unstable cone} at $x^*$
    \[
      C^u
      =
      \Big\{x_0 \in \R^2 :
      \sin \big|\angle\big(\tfrac{x_0-x^*}{\|x_0-x^*\|}, e_2\big)\big| \geq 2 e^{-\frac{\eps}{4} T}
      \Big\}
    \]
    contains a piece of the unstable manifolds
    $W^u = \{x_0 \in U: \lim \limits_{t \to -\infty} \varphi(t,x_0) = x^*\}$ at $x^*$
    \[
      W^u \cap B(x^*,\delta(T)) \subset C^u
    \]
    and $H(x) \in [0,\eps)$ for all $x \in C^u \cap B(x^*,\delta(T))$.

  \item[(iii)] \emph{Stable cone contains stable manifold and has maximal FTME values:}
    The so-called \emph{stable cone} at $x^*$
    \[
      C^s
      =
      \Big\{x_0 \in \R^2 :
      \sin \big|\angle\big(\tfrac{x_0-x^*}{\|x_0-x^*\|}, e_2\big)\big| \leq \tfrac{\eps}{4} e^{(\lambda_2-\lambda_1) T}
      \Big\}
    \]
    contains a piece of the stable manifolds
    $W^s = \{x_0 \in U : \lim \limits_{t \to \infty} \varphi(t,x_0) = x^*\}$ at $x^*$
    \[
      W^s \cap B(x^*,\delta(T)) \subset C^s
    \]
    and $H(x) > \lambda_1 - \eps$ for all $x \in C^s \cap B(x^*,\delta(T))$. Moreover, along the stable manifold
    $H(x) \in (\lambda_1 - \lambda_2 - \eps, \lambda_1 - \lambda_2 + \eps)$ for all $x \in W^s \cap B(x^*,\delta(T))$.
\end{itemize}
\end{theorem}


\begin{proof}
(i) By assumption, $\Phi_{x^*}(T) e_{i} = e^{\lambda_{i} T} e_{i}$ for $i=1,2$. For $x \in U$ let $v_1(x,T), v_2(x,T)$ denote
the normalized singular vectors of $\Phi_{x}(T)$ w.r.t.\ the singular values $e^{\lambda_1(x,T) T}$, $e^{\lambda_2(x,T) T}$, i.e.\ $\Phi_{x}(T)^{\rT}\Phi_{x}(T)v_{i}(x,T) = e^{2\lambda_{i}(x,T) T} v_{i}(x,T)$ and $v_1(x,T) \perp v_2(x,T)$. Moreover, it follows that $\lim_{T \to \infty} \lambda_{i}(x^*,T) = \lambda_{i}$.

Choose and fix $\eps \in (0, \min\{\lambda_1, -\lambda_2\})$. Then there exists $T(\eps)>0$ such that for $i=1,2$
\begin{equation}\label{eq09}
  \big| \lambda_{i}(x^*,T) - \lambda_{i} \big|
  \leq
  \tfrac{\epsilon}{16}
  \qquad \text{for all }
  T \geq T(\eps)
  .
\end{equation}
With the abbreviations $v_i = v_i(x,T)$ and using the fact that $e_2 = \langle e_2, v_1\rangle v_1 + \langle e_2, v_2\rangle v_2$ with $\langle e_2, v_1\rangle^2 + \langle e_2, v_2\rangle^2 = 1$, it follows that
\begin{eqnarray*}
  1
  &=&
  e^{-2\lambda_2 T} \|\Phi_{x^*}(T)e_2 \|^2
  =
  e^{-2\lambda_2 T} \big\langle \Phi_{x^*}(T)^{\rT}\Phi_{x^*}(T) e_2,e_2 \big\rangle
\\[0.5ex]
  & = &
  \langle e_2, v_1\rangle^2 e^{2 (\lambda_1(T,x^*) - \lambda_2)T} +
  \langle e_2, v_2\rangle^2 e^{2 (\lambda_2(T,x^*)-\lambda_2) T}
\end{eqnarray*}
and consequently
\begin{eqnarray*}
  \langle e_2, v_1\rangle^2
  & = &
  \|e_2\|^2 \|v_1\|^2 \cos^2 \angle(e_2,v_1)
  =
  \sin^2 \angle (e_2,v_2)
\\[0.5ex]
  & \leq &
  e^{2 (\lambda_2 - \lambda_1(x^*,T))T}
  \leq
  e^{2 (\lambda_2 - \lambda_1 + \frac{\epsilon}{16})T} < e^{-\frac{\epsilon}{2}T}
  ,
\end{eqnarray*}
which implies that $ |\sin \angle (e_2,v_2)| < e^{-\frac{\eps}{4}T}$.

The directional growth rate $\alpha(x) = \alpha(x,T,f(x))$ in \eqref{FTMEfield} satisfies $\alpha(x) \in [\lambda_2(x), \lambda_1(x)]$ and hence $\big(\lambda_1(x,T) - \alpha(x)\big)^+ = \lambda_1(x,T) - \alpha(x)$ and $\big(\lambda_2(x,T) - \alpha(x)\big)^+ = 0$. By enlarging $T(\eps)$ if necessary, it therefore follows from \eqref{approx} in Theorem \ref{thm2} that
\begin{equation}\label{eq04}
  \big| H(x) - (\lambda_1(x,T) - \alpha(x)) \big|
  \leq \tfrac{\eps}{8}
  \qquad \text{for all }
  x \in U, T \geq T(\eps)
  .
\end{equation}

Using the continuity of $x \mapsto \lambda_1(x,T)$ we choose for every $T \geq T(\eps)$ a $\delta = \delta(T) > 0$ such that $B(x^*,\delta) \subset U$ and
\[
  \big| \lambda_1(x,T) - \lambda_1(x^*,T) \big|
  \leq
  \tfrac{\eps}{16}
  \qquad \text{for all }
  x \in B(x^*,\delta)
  ,
\]
which implies that for $x \in B(x^*,\delta)$
\begin{equation}\label{eq06}
  \big| \lambda_1(x,T) - \lambda_1 \big|
  \leq
  \big|\lambda_1(x,T) - \lambda_1(x^*,T) \big| +
  \big| \lambda_1 - \lambda_1(x^*,T) \big|
  \leq
  \tfrac{\eps}{16} + \tfrac{\eps}{16}
  =
  \tfrac{\eps}{8}.
\end{equation}

By enlarging $T(\eps)$ and shrinking $\delta = \delta(T) > 0$ for $T \geq T(\eps)$ if necessary, it follows from  Lemma \ref{lemma1} that
\begin{equation}\label{eq07}
  \bigg| \alpha(x) - \frac{1}{T} \log \frac{\|\Phi_{x^*}(T) D_x f(x^*)(x-x^*)\|}
  {\|D_x f(x^*)(x-x^*)\|} \bigg|
  \leq
  \frac{\eps}{8}
  \qquad \text{for all }
  x \in B(x^*,\delta)
  .
\end{equation}
Combining \eqref{eq04}, \eqref{eq06} and \eqref{eq07}, we conclude that
\begin{equation}\label{eq10}
  \bigg|
  H(x) - \lambda_1 + \frac{1}{T} \log \frac{\|\Phi_{x^*}(T)(x-x^*)\|}{\|x-x^*\|}
  \bigg|
  \leq
  \frac{\eps}{2}
  \qquad \text{for }
  T \geq T(\eps),
  x \in B(x^*,\delta(T))
  .
\end{equation}
This, together with the estimate $\frac{1}{T} \log \frac{\|\Phi_{x^*}(T)(x-x^*)\|}{\|x-x^*\|} > \lambda_2(x^*,T) > \lambda_2 - \frac{\eps}{16}$, implies
\[
  H(x)
  <
  \lambda_1 -
  \frac{1}{T} \log \frac{\|\Phi_{x^*}(T)(x-x^*)\|}{\|x-x^*\|} + \frac{\eps}{2} < \lambda_1 - \lambda_2 + \eps,
\]
proving $H(x) \in [0,\lambda_1 - \lambda_2 + \eps)$ for all $x \in B(x^*,\delta(T))$.

(ii) By enlarging $T(\eps)$ if necessary, we ensure that for all $T \geq T(\eps)$ the estimates $\sin|\angle(e_1,e_2)| > 2 e^{-\frac{\eps}{4}T} > \frac{\eps}{4}e^{(\lambda_2-\lambda_1)T}$ hold and hence $e_1 \in C^u$. The unstable manifold $W^u$ is tangential to $e_1$ in $x^*$, i.e.
\[
  \lim_{\substack{x \to x^*,\\ x \in W^u}} \frac{x - x^*}{\|x - x^*\|} = e_1
  .
\]
By shrinking $\delta = \delta(T) > 0$ for $T \geq T(\eps)$ if necessary, we can therefore ensure that $W^u \cap B(x^*,\delta(T)) \subset C^u$.

Let $x \in W^u \cap B(x^*,\delta(T))$, $x \neq x^*$. Define $v = \frac{x-x^*}{\|x-x^*\|}$. Then $|\sin \angle (v,e_2)| \geq 2 e^{-\frac{\epsilon}{4}T}$. Since
\[
|\sin \angle (v,e_2)| = |\sin( \angle (v,v_2) + \angle (v_2,e_2))| \leq |\sin \angle (v,v_2)| + |\sin \angle (v_2,e_2)|,
\]
it follows that $|\langle v,v_1\rangle| = |\cos \angle(v,v_1)| =|\sin \angle (v,v_2)| \geq  e^{-\frac{\epsilon}{4}T}$. We then have
\[
e^{2\lambda_1(x^*,T)T} \geq \|\Phi_{x^*}(T)v\|^2 = \langle v,v_1\rangle^2 e^{2\lambda_1(x^*,T)T} +  \langle v,v_2\rangle^2 e^{2\lambda_2(x^*,T)T}  \geq e^{2(\lambda_1(x^*,T)-\frac{\epsilon}{4})T}.
\]
Hence $|\frac{1}{T} \log \|\Phi_{x^*}(T)v\| - \lambda_1(x^*,T)| \leq \frac{\epsilon}{4}$. Combining with \eqref{eq09} and \eqref{eq10} it follows that $H(x_0) < \epsilon$ and (b) is proved.

(iii) As in (ii) we can ensure that $W^s \cap B(x^*,\delta(T)) \subset C^u$ by shrinking $\delta(T)>0$ if necessary. Let $x \in W^s \cap B(x^*,\delta(T))$, $x \neq x^*$, and define $v = \frac{x-x^*}{\|x-x^*\|}$. Then $|\sin \angle (v,e_2)| \leq \frac{\eps}{4} e^{(\lambda_2 - \lambda_1)T}$. Since $\eps < \min \{- \lambda_2,\lambda_1\}$, we enlarge $T(\eps)$ is necessary, to ensure that for all $T \geq T(\eps)$ the estimate $\frac{\eps}{4} > e^{(\lambda_2 + \frac{\eps}{8})T}$ holds. We have
\begin{eqnarray*}
  | \sin \angle (v,e_2) |
  & < &
  e^{(-\lambda_1 - \frac{\eps}{16})T}
  (\tfrac{\eps}{2} - e^{(\lambda_2+\frac{\eps}{8})T})
\\
  & < &
  e^{(-\lambda_1 - \frac{\eps}{16})T}
  (e^{\frac{\eps}{2}T}-1 - e^{(\lambda_2+\frac{\eps}{8})T})
\\
  & < &
  e^{(-\lambda_1 - \frac{\eps}{16})T} (e^{\frac{\eps}{2}T}-1) -
  e^{(\lambda_2-\lambda_1 +\frac{\eps}{16})T}
  ,
\end{eqnarray*}
hence
\begin{eqnarray*}
  |\langle v,v_1\rangle |
  & = &
  |\sin \angle (v,v_2)| < |\sin \angle (v,e_2)| + |\sin \angle (v_2,e_2)|
\\
  & < &
  e^{(-\lambda_1 - \frac{\eps}{16})T}
  (e^{\frac{\eps}{2}T}-1) < (e^{\frac{\eps}{2}T}-1) e^{-\lambda_1(T,x^*)T}
  .
\end{eqnarray*}
It follows that
\begin{eqnarray*}
  \|\Phi_{x^*}(T)v\|^2
  & = &
  \langle v, v_1 \rangle^2 e^{2\lambda_1(x^*,T)T} + \langle v, v_2 \rangle^2 e^{2\lambda_2(x^*,T)T}
\\
  & \leq &
  (e^{\frac{\eps}{2}T} - 1)^2 + 1 < e^{\eps T}
  ,
\end{eqnarray*}
or equivalently $\frac{1}{T}\log \|\Phi_{x^*}(T)v\| < \frac{\eps}{2}$. Therefore
\[
  H(x)
  >
  \lambda_1 - \frac{1}{T}\log \|\Phi_{x^*}(T)v\| -
  \frac{\eps}{2} > \lambda_1 - \eps
  \qquad \text{for all }
  x \in C^s \cap B(x^*,\delta(T))
  .
\]
By enlarging $T(\eps)$ and shrinking $\delta(T) > 0$ if necessary, we can ensure that for $T \geq T(\eps)$ and $x \in M^s \cap B(x^*,\delta(T))$
\begin{equation}\label{eq11}
  \bigg| \frac{1}{T} \log \frac{\|\Phi_{x^*}(T)(x-x^*)\|}{\|x-x^*\|} -
  \frac{1}{T} \|\Phi_{x^*}(T) e_2\| \bigg|
  =
  \bigg| \frac{1}{T} \log \frac{\|\Phi_{x^*}(T)(x-x^*)\|}{\|x-x^*\|} - \lambda_2 \bigg|
  \leq
  \frac{\eps}{2}
  .
\end{equation}
Combining \eqref{eq10} and \eqref{eq11}, we get
\[
  H(x)
  >
  \lambda_1 -
  \frac{1}{T} \log \frac{\|\Phi_{x^*}(T)(x-x^*)\|}{\|x-x^*\|} -
  \frac{\eps}{2} > \lambda_1 - \lambda_2 - \eps
  \qquad \text{for all }
  x \in M^s \cap B(x^*,\delta(T))
  ,
\]
proving (iii).
\end{proof}

Figures \ref{fig:1FTME3D} and \ref{fig:2FTME3D} illustrate the statement of Theorem \ref{autonomentropy} for the Examples \ref{exlinear} and \ref{parabola}. Note that Theorem \ref{autonomentropy} states the existence of ridge and trough-like structures of order zero, i.e.\ described by values of the weighted FTME field $H(x)$ and not by conditions which also utilize  its first and second order derivatives (cp.\ also the different ridge notions in \cite{eberly-etal1994}).

\begin{figure}[ht!]
\centering
\begin{minipage}[b]{0.45\linewidth}
\centerline{\includegraphics[width=10cm]{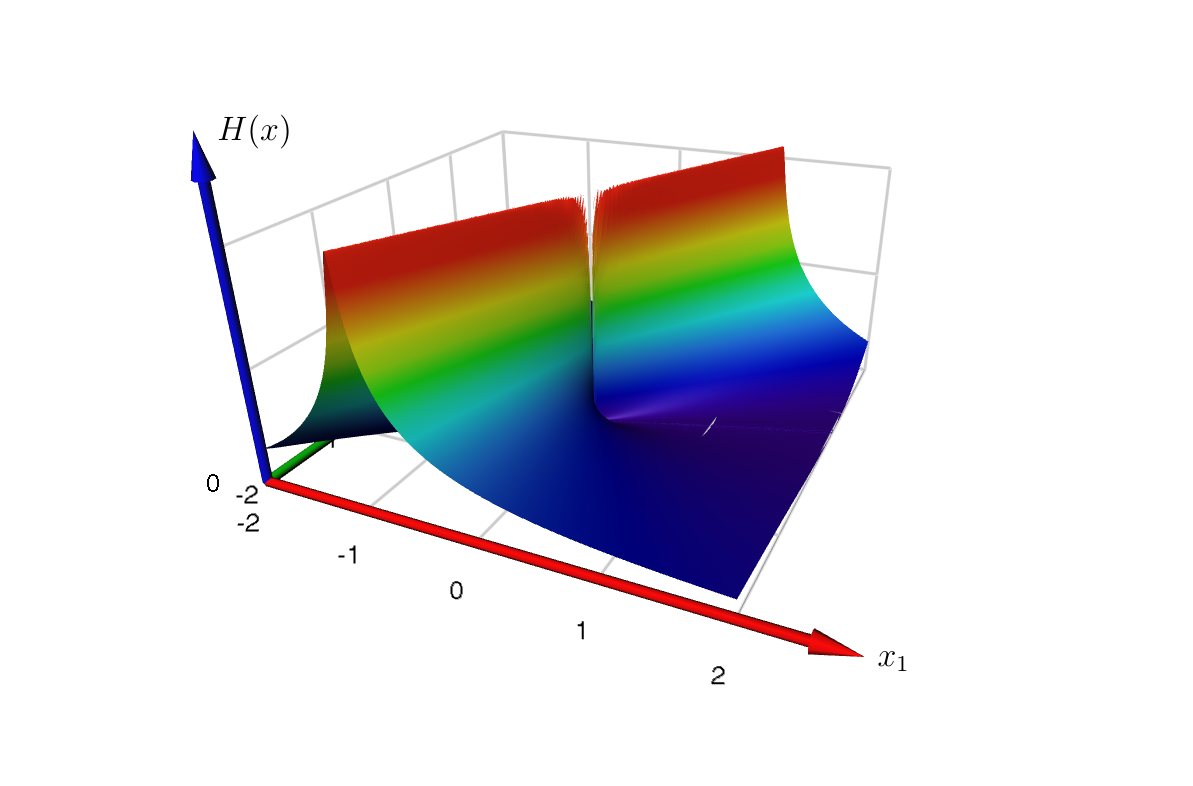}}
\caption{Weighted FTME field \eqref{FTMEfield} as depicted in Figure \ref{fig:linear2}.}
\label{fig:1FTME3D}
\end{minipage}
\hfill
\begin{minipage}[b]{0.45\linewidth}
\centerline{\includegraphics[width=10cm]{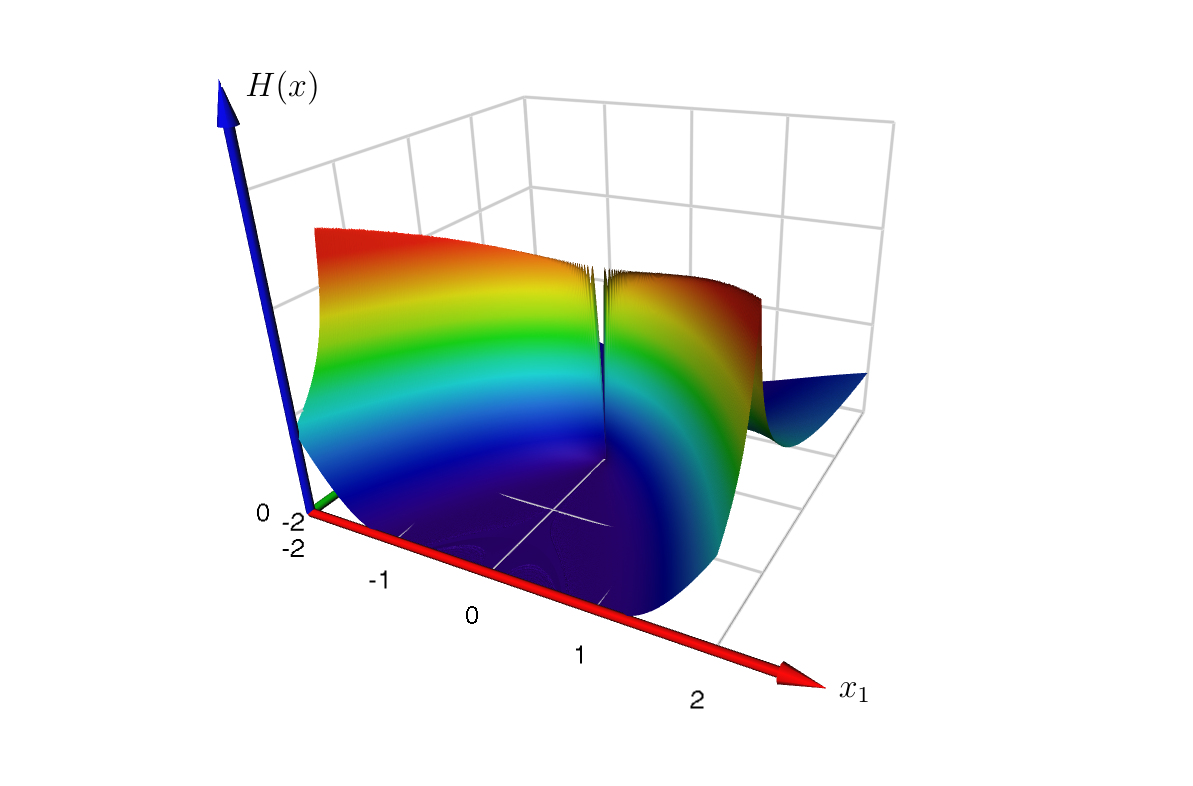}}
\caption{Weighted FTME field \eqref{FTMEfield} as depicted in Figure \ref{fig:FTME}.}
\label{fig:2FTME3D}
\end{minipage}
\end{figure}

\begin{remark}\label{rem:LCS}
An important feature of Lagrangian coherent structures which ensures objectivity in the sense of \emph{frame-independence} is formulated as invariance under time-dependent transformations of the form $y = Q(t)x + a(t)$, where $y$ denotes the new variable, $Q(t)$ is an orthogonal matrix and $a(t)$ is a translation vector (see e.g.\ \cite{haller2001}). It follows directly from Corollary \ref{thm-ellipsoid} that FTME is frame-independent, because FTME is characterized by the Lebesgue measure of the intersection of ellipsoids with balls and the volume of the intersection does not change under rotations and translations. However, the weight $\alpha(x,T,f(x))$ of the weighted FTME field \eqref{FTMEfield} depends on $x$ and the vector field $f(x)$ and is therefore in general not frame-independent. It is chosen such that it emphasizes the role of the equilibrium and its stable and unstable manifolds which occur as ridge and trough-like structures.
\end{remark}


\section*{Acknowledgments}

This work was partially supported by the German Science Foundation DFG under grant Si801-6 and DFG excellence cluster cfAED, and by Vietnam National Foundation for Science and Technology Development (NAFOSTED) under grant number 101.02-2011.47. The authors thank Tino Weinkauf for providing the figures and Gary Froyland, George Haller and Kathrin Padberg-Gehle for interesting discussions.


\end{document}